\documentclass[oneside,10pt]{article}          

\usepackage[a4paper]{geometry}	    
\usepackage{amsfonts,amsmath,latexsym,amssymb} 
\usepackage{amsthm}                
\usepackage{mathrsfs,upref}         
\usepackage{mathptmx}		    

\usepackage{color}

\newtheorem{theorem}{Theorem}[section]         
\newtheorem{lemma}[theorem]{Lemma}

\newtheorem{proposition}[theorem]{Proposition}

\theoremstyle{definition}

\newtheorem{definition}[theorem]{Definition}
\newtheorem{example}[theorem]{Example}

\newtheorem{remark}[theorem]{Remark}

\numberwithin{equation}{section}

\newcommand{\Skindef}{[\,\cdot,\cdot\,]}

\newcommand{\re}{\mathcal{R}}

\newcommand{\cA}{\mathcal A}
\newcommand{\cD}{\mathcal D}

\newcommand{\cH}{\mathcal H}

\newcommand{\cL}{\mathcal L}
\newcommand{\CC}{\mathbb C}
\newcommand{\NN}{\mathbb N}
\newcommand{\RR}{\mathbb R}
\newcommand{\fra}{\mathfrak{a}}

\newcommand{\frd}{\mathfrak{d}}

\newcommand{\frt}{\mathfrak{t}}
\newcommand{\defeq}{\mathrel{\mathop:}=}
\newcommand{\defequ}{\mathrel{\mathop:}\hspace*{-0.5ex}&=}
\newcommand{\eqdef}{=\mathrel{\mathop:}}
\newcommand{\sess}{\sigma_{\rm ess}}

\DeclareMathOperator{\Real}{Re}
\DeclareMathOperator{\Imag}{Im}
\renewcommand{\Re}{\Real}
\renewcommand{\Im}{\Imag}
\DeclareMathOperator{\ran}{ran}

\newcommand{\dom}{\mathcal{D}}

\newcommand{\hhalf}{H_{\frac{1}{2}}}
\newcommand{\half}{\frac{1}{2}}
\newcommand{\mphhalf}{H_{-\frac{1}{2}}\times H_{\frac{1}{2}}}
\newcommand{\eps}{\varepsilon}
\newcommand{\wh}{\widehat}
\newcommand{\wt}{\widetilde}
\newcommand{\rd}{\mathrm{d}}
\newcommand{\dmin}{d_{\rm min}}
\newcommand{\dmax}{d_{\rule{0ex}{1.2ex}\rm max}}
\newcommand{\frdmin}{\mathfrak{d}_{\rm min}}
\newcommand{\frdmax}{\mathfrak{d}_{\rm max}}
\newcommand{\frtmin}{\mathfrak{t}_{\rm min}}
\newcommand{\frtmax}{\mathfrak{t}_{\rm max}}
\newcommand{\Tmin}{T_{\rm min}}
\newcommand{\Tmax}{T_{\rm max}}
\newcommand{\pmin}{p^{({\rm min})}}
\newcommand{\pmax}{p^{({\rm max})}}
\newcommand{\Nmin}{N_{\rm min}}
\newcommand{\Nmax}{N_{\rm max}}
\newcommand{\lambdamin}{\lambda^{({\rm min})}}
\newcommand{\lambdamax}{\lambda^{({\rm max})}}
\newcommand{\vect}[2]{\begin{pmatrix} #1 \\ #2 \end{pmatrix}}
\newcommand{\smvect}[2]{\bigl(\begin{smallmatrix} #1 \\ #2 \end{smallmatrix}\bigr)}

\newcommand{\be}{\begin{equation}}
\newcommand{\ee}{\end{equation}}
\newcommand{\bea}{\begin{eqnarray}}
\newcommand{\eea}{\end{eqnarray}}
\newcommand{\beann}{\begin{eqnarray*}}
\newcommand{\eeann}{\end{eqnarray*}}

\renewcommand{\re}{{\rm Re}}

\newcommand\void[1]{}

\newcommand{\myassumption}[2]{
\begin{list}{}{\setlength{\leftmargin}{7ex}
\setlength{\topsep}{1ex}
\setlength{\labelwidth}{7ex}
\setlength{\labelsep}{2ex}
}
\item[{\rm #1}]
#2
\end{list}
}

\newcommand{\myassumptionbf}[2]{
\begin{list}{}{\setlength{\leftmargin}{8ex}
\setlength{\labelwidth}{8ex}
\setlength{\labelsep}{2ex}
}
\item[\textbf{(#1)}]
#2
\end{list}
}

\title{Variational principles
for self-adjoint operator functions arising from second-order systems}

\author{Birgit Jacob\thanks{Fachbereich C -- Mathematik und Naturwissenschaften,
Arbeitsgruppe Funktionalanalysis,
Bergische Universit\"at  Wuppertal,
Gau\ss stra\ss e 20,
D-42119 Wuppertal, Germany, {bjacob@uni-wuppertal.de}}, \and Matthias Langer\thanks{Department of Mathematics and Statistics,
University of Strathclyde,
26 Richmond Street,
Glasgow  G1 1XH, United Kingdom, {m.langer@strath.ac.uk}}
 \and Carsten Trunk\thanks{Institut f\"ur Mathematik,
Technische Universit\"at Ilmenau,
Postfach 100565,
 D-98684 Ilmenau, Germany,
{carsten.trunk@tu-ilmenau.de}}}

\date{}                               

\begin{document}
\maketitle

\begin{abstract}
Variational principles are proved for self-adjoint operator functions arising from   variational evolution equations of the form
\begin{equation*}
  \langle\ddot{z}(t),y \rangle + \frd[\dot{z} (t), y] + \fra_0 [z(t),y] = 0.
\end{equation*}
Here $\fra_0$ and $\frd$ are  densely defined, symmetric and positive sesquilinear forms on a Hilbert space $H$.
We associate with the variational evolution equation an equivalent Cauchy problem corresponding to a block operator matrix $\mathcal{A}$, the forms
\begin{equation*}
  \frt(\lambda)[x,y] \defeq \lambda^2\langle x,y\rangle + \lambda\frd[x,y] + \fra_0[x,y],
\end{equation*}
where $\lambda\in \mathbb C$ and $x,y$ are in the domain of the form $\fra_0$, and a corresponding operator family $T(\lambda)$.
Using form methods we define a generalized Rayleigh functional  and characterize the eigenvalues
above the essential spectrum of $\mathcal{A}$  by a min-max
and a max-min variational principle.
The obtained results are illustrated with a damped beam equation.
%
%
%
\\[1ex]
\\[1ex]
\end{abstract}



\section{Introduction}

Variational principles are a very useful tool for the qualitative and numerical investigation
of eigenvalues of self-adjoint operators and operator functions.
For instance, the eigenvalues
$\lambda_1 \ \leq \lambda_2 \leq  \ldots $
below the essential spectrum of a self-adjoint operator $A$ that is bounded
from below and has domain $\dom(A)$ can be characterized using the Rayleigh functional
\begin{equation*}
  p(x) = \frac{\langle Ax,x \rangle}{\langle x,x \rangle}\,,
  \qquad x\in \dom(A), \; x\ne 0,
\end{equation*}
 via a min-max principle or a max-min principle:
 \begin{equation*}
  \lambda_n = \min_{\substack{L\subset\dom(A) \\ \dim L =n}}
  \;\;\max_{x\in L\setminus\{0\}}\;\; p(x)
  = \max_{\substack{L\subset H \\ \dim L = n-1}}
  \;\; \min_{\substack{x\in\dom(A) \setminus\{0\} \\ x \perp L}}
  \; p(x).
\end{equation*}
Variational principles were first introduced by H.~Weber, Lord Rayleigh, H.~Poincar\'{e},
E.~Fischer, G.~Polya, and W.~Ritz, H.~Weyl, R.~Courant
(see, e.g.\ \cite{BEL00,D55,wein}, and the references therein).

In this article we investigate variational principles for self-adjoint operator functions
arising from   variational evolution equations of the form
\begin{equation}\label{vareqn}
  \langle\ddot{z}(t),y \rangle + \frd[\dot{z} (t), y] + \fra_0 [z(t),y] = 0.
\end{equation}
Here $\fra_0$ with domain $\cD(\fra_0)$ and $\frd$ with domain $\cD(\frd) \supset \cD(\fra_0)$ are  densely defined, symmetric and posivite sesquilinear forms on a Hilbert space $H$ satisfying
\textbf{(F1)}--\textbf{(F3)}, see Section~\ref{sec:framework}.
With  this variational evolution equation we associate  a Cauchy problem
\begin{equation}\label{Cauchy}
  \vect{\dot{z}}{\dot{w}}= \cA \vect{z}{w}, \qquad
  \vect{z(0)}{w(0)} = \vect{z_0}{w_0}
\end{equation}
on $\cD(\fra_0)\times H$ in such a way that the solutions of \eqref{vareqn} equal the first component of the solutions of \eqref{Cauchy}.
For $\lambda\in \mathbb C$ we define  the sesquilinear form
\begin{equation}\label{form}
  \frt(\lambda)[x,y] \defeq \lambda^2\langle x,y\rangle + \lambda\frd[x,y] + \fra_0[x,y]
\end{equation}
with domain $\cD(\frt(\lambda)): = H_{\frac{1}{2}}:= \cD(\fra_0)$.
We identify a disc $\Phi_{\gamma_0}\subset\CC$
which is the largest disc around zero with an empty intersection
with the essential spectrum of $\cA$.
For $\lambda \in \Phi_{\gamma_0}$ we show that the form
$\frt(\lambda)$ is closed and sectorial and that the corresponding
operator $T(\lambda)$ is $m$-sectorial. Moreover, on $\Phi_{\gamma_0}$
the spectrum (point spectrum) of $\cA$ and the spectrum (resp.\ point spectrum) of $T$ coincide.

In \cite{D55} R.\,J.~Duffin proved a variational principle for eigenvalues
of a quadratic matrix polynomial, which was generalized in various directions
to more general operator functions; see, e.g.\ the references in \cite{EL} and \cite{voss15}.
In \cite{EL} such a variational principle was proved for eigenvalues
of operator functions whose values are possibly unbounded self-adjoint operators.
Here we adapt this variational principle from \cite{EL} to our situation.
Using the form $\frt(\lambda)$ we introduce a slightly
more general definition of a generalized Rayleigh functional and we show that
the variational principle generalizes to this situation.
In particular, for a fixed  $x\in H_{\frac{1}{2}}\setminus \{0\}$,
denote the two real solutions (if they exist) of the quadratic equation
\begin{equation*}
  \frt(\lambda) [x,x] =0
\end{equation*}
by $p_-(x)$ and $p_+(x)$ such that $p_-(x)\leq p_+(x)$ is satisfied
and set $p_+(x)\defeq-\infty$, $p_-(x)\defeq\infty$ if there are no real solutions.
Then the function $p_+$ plays the role of a generalized Rayleigh functional in
our main theorem, which yields variational principles for the real eigenvalues of
$\cA$ or, what is equivalent, of $T$. These variational principles
hold in certain real intervals
$\Delta$ above the essential spectrum of $\cA$ in the disc $\Phi_{\gamma_0}$
with the property that $\Delta$ does not contain values of $p_-$.
In $\Delta$ the spectrum of
$\cA$ is either empty or consists only
of a finite or infinite sequence of isolated semi-simple eigenvalues of finite multiplicity
of $\cA$.  Moreover, we show that these eigenvalues $\lambda_1 \ge \lambda_2 \ge \cdots$,
counted according to their multiplicities, satisfy
\begin{equation*}
  \lambda_n = \max_{\substack{L\subset H_{1/2} \\ \dim L = n}}
  \;\; \min_{x\in L\setminus\{0\}}
  \;\; p_+(x)
  =  \min_{\substack{L\subset H \\ \dim L = n-1}}
  \;\; \sup_{\substack{x\in H_{1/2}\setminus\{0\} \\ x \perp L}}
  \; p_+(x)
\end{equation*}
and, if $N<\infty$, we show for $n>N$ that
\begin{equation*}
    \sup_{\substack{L\subset\cD \\ \dim L = n}}
    \;\;\min_{x\in L\setminus\{0\}}\;\; p_+(x)
    \le \inf\Delta \qquad \text{and} \qquad
    \inf_{\substack{L\subset H \\ \dim L = n-1}}
    \;\; \sup_{\substack{x\in\cD\setminus\{0\} \\ x \perp L}}\; p_+(x)
    \le \inf\Delta.
\end{equation*}
A major application of this variational principle is a quite
general interlacing principle  which is the second
main result of this article:
if the stiffness operator $A_0$ decreases and the damping operator
$D$ increases, then the corresponding $n$th eigenvalue
decreases compared with the $n$th eigenvalue
of the unchanged system.
We illustrate the obtained results with an example where we consider
a beam equation with a damping such that $A_0$ corresponds to the fourth
derivative on the interval $(0,1)$
(with some appropriate boundary conditions) and
the damping $D$ equals $-\frac{\rd}{\rd x} d \frac{\rd}{\rd x}$ with
some smooth function $d$ (and some boundary conditions).

We proceed as follows.  The variational principle
obtained in \cite{EL} is adapted to the setting of this paper in Section~\ref{sec:var}.
Section~\ref{sec:framework} is devoted to
general properties of the class of second-order systems studied in this paper.
The main results of this paper are proved in Section~\ref{sec:poly}.
In particular, we study the form \eqref{form} and their relation to
the operator matrix $\cA$ and the operator function $T(\lambda)$.
On a disc $\Phi_{\gamma_0}$ around zero, $\frt(\lambda)$ is a
closed sectorial form and the spectrum (point spectrum) of $\cA$ and
the spectrum (point spectrum) of $T$ coincide.
Further, the variational principles for $\cA$ are presented in
Theorem~\ref{theo100}. As an application of the variational principle
we show interlacing properties of eigenvalues of two different second-order
problems with coefficients which satisfy a specific order relation.
Finally, in Section~\ref{sec:example} we apply the obtained results to a damped beam equation.

Throughout this  paper we use the following notation.
For a self-adjoint operator $S$ and an interval $I$ we denote by $\cL_I(S)$
the spectral subspace of $S$ corresponding to $I$.
A closed, densely defined  operator in $H$ is called \emph{Fredholm} if the
dimension of its kernel
and the (algebraic) co-dimension of its range are finite.
The \emph{essential spectrum} of a closed, densely defined
 operator $S$ is defined by
\begin{displaymath}
  \sess(S) \defeq \bigl\{\lambda \in {\mathbb C}\; |\; S-\lambda I \;\;
  \text{is not Fredholm}\bigr\}.
\end{displaymath}
A closed, densely defined operator $T$ is called
\emph{sectorial} if its numerical range is contained in a sector
$\{ z\in \mathbb C \mid \Re z\ge z_0, \;  |\arg (z -z_0)| \leq \theta\}$
for some $z_0 \in \mathbb R$ and $\theta \in [0,\frac{\pi}{2})$.
A sectorial operator $T$ is called  \emph{m-sectorial} if $\lambda
\in \rho(T)$ for some $\lambda$ with $\Re \lambda < z_0$;
see, e.g.\ \cite[\S V.3.10]{K}.
For a sesquilinear form $\fra[\,\cdot\,,\cdot\,]$ with domain $\cD(\fra)$
the corresponding quadratic form is defined by $\fra[x]\defeq\fra[x,x]$, $x\in\cD(\fra)$.
A form is called \emph{sectorial} if its numerical range is contained in a sector
$\{ z\in \mathbb C \mid \Re z\geq z_0, \;  |\arg (z -z_0)| \leq \theta\}$
for some $z_0 \in \mathbb R$ and $\theta \in [0,\frac{\pi}{2})$;
see, e.g.\ \cite[\S V.3.10]{K}.

\section{A general variational
principle for self-adjoint operator functions}
\label{sec:var}

In this section we recall a general variational principle for eigenvalues
of a self-adjoint operator function from \cite{EL}  adapted to the present situation.
Here we also show  some additional statements.  We mention that
in \cite{EL} a more general class of operator functions was investigated.
%
%
%
%

For the rest of this section let $\Delta\subset\RR$ be an interval with
\begin{equation}\label{Hitchin}
  a=\inf\Delta \quad\text{and}\quad b=\sup\Delta, \qquad
  -\infty \leq a <b\leq \infty,
\end{equation}
and let $\Omega$ be a domain in $\mathbb C$ such that
$\Delta\subset\Omega$.  On $\Omega$ we consider a family of closed, densely defined
operators $T(\lambda)$, $\lambda\in\Omega$, in a Hilbert space $H$
with inner product $\langle\,\cdot\,,\cdot\rangle$,
where $T(\lambda)$ has domain $\cD(T(\lambda))$.
In the following we shall assume that either $T(\lambda)$ or $-T(\lambda)$
is an m-sectorial operator for $\lambda\in\Omega$.
Under this assumption the sesquilinear form $\langle T(\lambda)\,\cdot\,,\cdot\rangle$
is closable for $\lambda\in\Omega$, and we denote the closure
by $\frt(\lambda)[\,\cdot\,,\cdot\,]$
with domain $\cD(\frt(\lambda))$ and set $\frt(\lambda)[x]\defeq\frt(\lambda)[x,x]$,
which is the corresponding quadratic form.
Recall (see, e.g.\ \cite[\S VII.4]{K}) that $T\defeq(T(\lambda))_{\lambda\in\Omega}$
is called a \emph{holomorphic family of type}~(B)
if $T(\lambda)$ is m-sectorial for $\lambda\in\Omega$, the domain
$\cD(\frt(\lambda))$ of the closed quadratic form $\frt(\lambda)$ is independent of $\lambda$,
which we denote by $\cD$, and $\lambda\mapsto\frt(\lambda)[x]$ is
holomorphic on $\Omega$ for every $x\in\cD$.

We suppose that one of the following two conditions is satisfied.

\myassumption{(I)}{
Let $\Omega$ be a domain in $\CC$ and $\Delta\subset\Omega\cap\RR$ an interval with
endpoints $a$, $b$ as in \eqref{Hitchin}.
The family $(T(\lambda))_{\lambda\in\Omega}$ is a holomorphic family of type (B),
$T(\lambda)$ is self-adjoint for $\lambda\in\Delta$ and there exists a $c\in\Delta$
such that $\dim\cL_{(-\infty,0)}(T(c))<\infty$.
}

\myassumption{(II)}{
Let $\Omega$ be a domain in $\CC$ and $\Delta\subset\Omega\cap\RR$ an interval
with endpoints $a$, $b$ as in \eqref{Hitchin}.
The family $(-T(\lambda))_{\lambda\in\Omega}$ is a holomorphic family of type (B),
$T(\lambda)$ is self-adjoint for $\lambda\in\Delta$ and there exists a $c\in\Delta$
such that $\dim\cL_{(0,\infty)}(T(c))<\infty$.
}

\noindent
Note that  under assumption (I) for $\lambda\in\Delta$ the
operators $T(\lambda)$ are self-adjoint and sectorial, and, hence,
bounded from below. Similarly, under assumption (II), the operators $T(\lambda)$
are bounded from above for $\lambda\in\Delta$.
The condition $\dim\cL_{(-\infty,0)}(T(c))<\infty$ is equivalent to the fact that
$\sigma(T(c))\cap(-\infty,0)$ consists of at most a finite number of eigenvalues of
finite multiplicities.

Before we formulate the second set of assumptions, let us recall the
following definitions.
The \emph{spectrum} of the operator function $T$ is defined as follows:
\begin{align*}
  \sigma(T) \defequ \bigl\{\lambda\in\Omega \mid T(\lambda)
  \text{ is not bijective from}\;\cD(T(\lambda))\;\text{onto}\; H \bigr\} \\[0.5ex]
  &= \bigl\{\lambda\in\Omega \mid 0\in\sigma(T(\lambda))\bigr\}.
\end{align*}
Similarly, the \emph{essential spectrum} of the operator function $T$ is defined as
\[
  \sess(T) \defeq \bigl\{\lambda\in\Omega \mid T(\lambda) \text{ is not Fredholm}\bigr\}
  = \bigl\{\lambda\in\Omega \mid 0\in\sess(T(\lambda))\bigr\}.
\]
A number $\lambda\in\Omega$ is called an \emph{eigenvalue} of the operator function $T$
if there exists an $x\in\cD(T(\lambda))$, $x\ne0$, such that $T(\lambda)x=0$.
The \emph{point spectrum} is the set of all eigenvalues:
\begin{align*}
  \sigma_{\rm p}(T) \defequ \bigl\{\lambda\in\Omega \mid
  \exists\,x\in\cD(T(\lambda)),\,x\ne0,\,T(\lambda)x=0\bigr\} \\[1ex]
  &= \bigl\{\lambda\in\Omega \mid 0\in\sigma_{\rm p}(T(\lambda))\bigr\},
\end{align*}
where $\sigma_{\rm p}(T(\lambda))$ denotes the point spectrum of the operator
$T(\lambda)$ for fixed $\lambda\in\Omega$.
The \emph{geometric multiplicity} of an eigenvalue $\lambda$ of the operator
function $T$ is defined as the dimension of $\ker T(\lambda)$.

In addition to (I) or (II) we shall assume that one of the following
two conditions $(\searrow)$, $(\nearrow)$ is satisfied.

\myassumption{$(\searrow)$}{
For every $x\in\cD\setminus\{0\}$ the function $\lambda\mapsto\frt(\lambda)[x]$ is
\emph{decreasing at value zero} on $\Delta$, i.e.\ if $\frt(\lambda_0)[x]=0$ for some
$\lambda_0\in\Delta$, then
\begin{align*}
  &\frt(\lambda)[x]>0 \qquad\text{for } \lambda\in(-\infty,\lambda_0)\cap\Delta, \\[0.5ex]
  &\frt(\lambda)[x]<0 \qquad\text{for } \lambda\in(\lambda_0,\infty)\cap\Delta.
\end{align*}
}

\myassumption{$(\nearrow)$}{
For every $x\in\cD\setminus\{0\}$ the function $\lambda\mapsto\frt(\lambda)[x]$ is
\emph{increasing at value zero} on $\Delta$, i.e.\ if $\frt(\lambda_0)[x]=0$ for some
$\lambda_0\in\Delta$, then
\begin{align*}
  &\frt(\lambda)[x]<0 \qquad\text{for } \lambda\in(-\infty,\lambda_0)\cap\Delta, \\[0.5ex]
  &\frt(\lambda)[x]>0 \qquad\text{for } \lambda\in(\lambda_0,\infty)\cap\Delta.
\end{align*}
}

If $T$ satisfies $(\nearrow)$ or $(\searrow)$, then,
for $x\in\cD\setminus\{0\}$,
the scalar function $\lambda\mapsto\frt(\lambda)[x]$ is either decreasing
or increasing at a zero and, hence, it has at most one zero in $\Delta$.

We now introduce the notion of a generalized Rayleigh functional $p$,
which is a mapping from $\cD\setminus\{0\}$ to $\RR\cup\{\pm\infty\}$.
If there is a zero $\lambda_0$ of the scalar
function $\lambda \mapsto \frt(\lambda)[x]$ in $\Delta$, then the corresponding value
of a generalized Rayleigh functional $p(x)$ must equal this zero; $p(x)=\lambda_0$.
Otherwise, there is some freedom in the definition.
More precisely, we use the following definition.

\begin{definition}\label{StJohn}
Let $\Delta$ and $\Omega$ be as above.  Moreover, let $T(\lambda)$, $\lambda\in\Omega$,
be a family of closed operators in a Hilbert space $H$ satisfying either
{\rm (I)} or {\rm(II)} and which satisfies also $(\nearrow)$  or $(\searrow)$.
In the case  $(\searrow)$ a mapping $p:\cD\setminus\{0\} \to \RR\cup\{\pm\infty\}$
with the properties
\begin{equation*}
  p(x) \begin{cases}
    = \lambda_0 & \text{if } \frt(\lambda_0)[x]=0, \\[1ex]
    < a & \text{if } a\in\Delta \text{ and } \frt(\lambda)[x]<0
      \text{ for all }\lambda\in\Delta, \\[1ex]
    \le a & \text{if } a\notin\Delta \text{ and } \frt(\lambda)[x]<0
      \text{ for all }\lambda\in\Delta, \\[1ex]
    > b & \text{if } b\in\Delta \text{ and } \frt(\lambda)[x]>0
      \text{ for all }\lambda\in\Delta , \\[1ex]
    \ge b & \text{if } b\notin\Delta \text{ and } \frt(\lambda)[x]>0
      \text{ for all }\lambda\in\Delta.
  \end{cases}
\end{equation*}
is called a \emph{generalized Rayleigh functional} for $T$ on $\Delta$.
In the case  $(\nearrow)$ a mapping $p:\cD\setminus\{0\} \to\RR\cup\{\pm\infty\}$
with the properties
\begin{equation}\label{def_p_incr}
  p(x) \begin{cases}
    = \lambda_0 & \text{if } \frt(\lambda_0)[x]=0, \\[1ex]
    > b & \text{if } b\in \Delta \text{ and } \frt(\lambda)[x]<0
      \text{ for all }\lambda\in\Delta, \\[1ex]
    \ge b & \text{if } b\notin \Delta \text{ and } \frt(\lambda)[x]<0
      \text{ for all }\lambda\in\Delta, \\[1ex]
    < a & \text{if } a\in \Delta \text{ and } \frt(\lambda)[x]>0
      \text{ for all }\lambda\in\Delta, \\[1ex]
    \le a & \text{if } a\notin \Delta \text{ and } \frt(\lambda)[x]>0
      \text{ for all }\lambda\in\Delta.
  \end{cases}
\end{equation}
is called a \emph{generalized Rayleigh functional} for $T$ on $\Delta$.
\end{definition}

\begin{remark}\label{Welshpool}
One possible choice for $p$ in the case $(\searrow)$ is the following (see \cite{BEL00,EL}).
For $x\in\cD\setminus\{0\}$ set
\[
  p(x) = \begin{cases}
    \lambda_0 & \text{if } \frt(\lambda_0)[x]=0, \\[1ex]
    -\infty & \text{if } \frt(\lambda)[x]<0 \text{ for all }\lambda\in\Delta, \\[1ex]
    +\infty & \text{if } \frt(\lambda)[x]>0 \text{ for all }\lambda\in\Delta,
  \end{cases}
\]
which was used as a definition of a generalized Rayleigh functional in \cite{BEL00,EL}.
However, here we propose to use the Definition
\ref{StJohn}. This  has the following advantage: if $p$ is a
generalized Rayleigh functional for
$T$ on $\Delta$, then  the same  $p$ remains a
generalized Rayleigh functional in the sense of Definition \ref{StJohn}
for $T$ on a smaller interval $\Delta'$ with
$\Delta'\subset \Delta$.
Moreover, in many applications, including the one in Section~\ref{sec:poly},
the operator function $T$ is defined on a larger interval $\tilde\Delta\supset\Delta$
but satisfies, say, $(\searrow)$ only on $\Delta$.
If $\frt(\cdot)[x]$ has a zero $\lambda_0$ in $\tilde\Delta$ where
$\lambda_0<a$ and $\frt(\lambda)[x]<0$ for all $\lambda\in\Delta$,
one can set $p(x)\defeq\lambda_0$.
\end{remark}

\begin{example}
We consider two examples to illustrate the notion of a generalized Rayleigh functional.
\begin{enumerate}
\item[(i)]
Let $A$ be a bounded self-adjoint operator in a Hilbert space $H$
and consider the operator function $T(\lambda)=A-\lambda I$, $\lambda\in\Omega=\CC$.
The corresponding quadratic forms are $\frt(\lambda)[x]=\langle Ax,x\rangle-\lambda\|x\|^2$,
$x\in\cD=H$.  If we take $\Delta=\RR$, then $T$ satisfies condition (I),
where one can choose any $c<\min\sigma(A)$; it also satisfies (II), where one can
choose any $c>\max\sigma(A)$.  Moreover, the function $T$ satisfies condition $(\searrow)$
since $\frt'(\lambda)[x]=-\|x\|^2$.  For each $x\in H\setminus\{0\}$ the
function $\frt(\cdot)[x]$ has the unique zero
\[
  p(x) = \frac{\langle Ax,x\rangle}{\|x\|^2}\,;
\]
hence the classical Rayleigh quotient is a generalized Rayleigh functional in
the sense of Definition~\ref{StJohn}.
\item[(ii)]
In $H=\CC^2$ consider the quadratic operator function
\[
  T(\lambda) = \begin{bmatrix} \lambda^2-2\lambda+1 & -2 \\[0.5ex]
  -2 & \lambda^2+1 \end{bmatrix}, \qquad \lambda\in\Omega\defeq\CC,
\]
and choose $\Delta\defeq(-\infty,0)$.
Clearly, conditions (I) and (II) are satisfied.
For $x=\binom{x_1}{x_2}\in\CC^2$ one has
\[
  \frt(\lambda)[x] = \langle T(\lambda)x,x\rangle
  = \|x\|^2\lambda^2 - 2|x_1|^2\lambda + \|x\|^2-4\Re(x_1\overline{x_2}).
\]
Since the coefficient of $\lambda$ is non-positive, the sum of the
two zeros of the polynomial $\frt(\cdot)[x]$ is non-negative if $x\ne0$,
and therefore at most one zero can be in $\Delta$.
At any such zero the function
must be decreasing, which shows that condition $(\searrow)$ is satisfied.
Moreover, $\frt(\cdot)[x]$ is positive on $\Delta$ if it has no negative zero.
Hence a possible choice for a generalized Rayleigh functional is given by
\[
  p(x) = \begin{cases}
    \displaystyle \frac{|x_1|^2-\sqrt{|x_1|^4-\|x\|^2+4\Re(x_1\overline{x_2})}\,}{\|x\|^2}
    & \text{if } |x_1|^4-\|x\|^2+4\Re(x_1\overline{x_2})\ge0, \\[1ex]
    \infty & \text{otherwise}.
  \end{cases}
\]
Note that three cases occur: (a) $\frt(\cdot)[x]$ has a positive and a negative
zero, in which case $p(x)$ equals the negative zero;
(b) $\frt(\cdot)[x]$ has two positive zeros, in which case $p(x)>0=\sup\Delta$;
(c) $\frt(\cdot)[x]$ has no real zeros, in which case $p(x)=\infty$.
Examples for these three cases are given by the vectors $\binom{1}{1}$,
$\binom{2}{-1}$, $\binom{1}{-1}$, respectively.
\end{enumerate}
\end{example}

\medskip

\noindent
For a generalized Rayleigh functional $p$ as in
Definition~\ref{StJohn} we have for $\lambda
\in \Delta$, $x\in\cD(T(\lambda))\setminus\{0\}$,
\[
  T(\lambda)x = 0 \quad \implies \quad p(x) =\lambda.
\]
If $T$ satisfies $(\searrow)$, then  for $x\in\cD\setminus\{0\}$
\begin{equation}\label{Trip1}
\begin{alignedat}{3}
  \frt(\lambda)[x] &> 0 \quad & &\iff \quad & p(x) &> \lambda, \\
  \frt(\lambda)[x] &< 0 \quad & &\iff \quad & p(x) &< \lambda;
\end{alignedat}
\end{equation}
if $T$ satisfies $(\nearrow)$, then for $x\in\cD\setminus\{0\}$
\begin{equation}\label{Trip2}
\begin{alignedat}{3}
  \frt(\lambda)[x] &> 0 \quad & &\iff \quad & p(x) &< \lambda, \\
  \frt(\lambda)[x] &< 0 \quad & &\iff \quad & p(x) &> \lambda.
\end{alignedat}
\end{equation}

In \cite[Theorem~2.1]{EL}  a variational principle involving a generalized
Rayleigh functional was derived.  There the generalized Rayleigh functional
was defined as in Remark~\ref{Welshpool} and not in the (slightly more general) way
as in Definition~\ref{StJohn}.  Therefore, the variational principle in the
following theorem is an adapted version of \cite[Theorem~2.1]{EL} where
a non-decreasing sequence of eigenvalues of an operator function is characterized.
Moreover, in \cite[Theorem~2.1]{EL} only the case {\rm (I), ($\searrow$)}
was considered (under slightly weaker assumptions on $\frt$).

\begin{theorem}\label{th:var_left}
Let $\Delta$ and $\Omega$ be as above.  Moreover, let $T(\lambda)$, $\lambda\in\Omega$,
be a family of closed operators in a Hilbert space $H$ satisfying either
{\rm (I), ($\searrow$)} or {\rm(II), ($\nearrow$)}, let $p$ be a generalized
Rayleigh functional and assume that
\[
  \Delta'\defeq \begin{cases}
  \Delta & \textit{if\, } \sess(T)\cap\Delta=\emptyset, \\[1ex]
  \bigl\{\lambda\in\Delta \mid \lambda<\inf\bigl(\sess(T)\cap\Delta\bigr)\bigr\}
  & \text{if\, } \sess(T)\cap\Delta\ne\emptyset,
  \end{cases}
\]
is non-empty.

Then $\sigma(T)\cap\Delta'$ is either empty or consists only of a finite or infinite
sequence of isolated eigenvalues of $T$ with finite geometric multiplicities, which
in the case of infinitely many eigenvalues in $\sigma(T)\cap\Delta'$ accumulates only
at $\sup\Delta'$ {\rm(}which equals $\inf(\sess(T)\cap\Delta)$
if $\sess(T)\cap\Delta\ne\emptyset$ and equals  $b$ otherwise{\rm)}.

If $\sigma(T)\cap\Delta'$ is empty, then set $N\defeq0$; otherwise, denote the
eigenvalues in $\sigma(T)\cap\Delta'$ by $(\lambda_j)_{j=1}^N$, $N\in\NN\cup\{\infty\}$,
in non-decreasing order, counted according to their geometric multiplicities:
$\lambda_1 \le \lambda_2 \le \cdots$.
Choose $a'\in\Delta'$ so that in the case $N>0$ it satisfies $a' \le \lambda_1$.
Then the quantity
\[
  \kappa \defeq \begin{cases}
    \dim\mathcal L_{(-\infty,0)}\bigl(T(a')\bigr) & \text{if\, {\rm (I), ($\searrow$)} are satisfied}, \\[1ex]
    \dim\mathcal L_{(0,\infty)}\bigl(T(a')\bigr) & \text{if\, {\rm (II), ($\nearrow$)} are satisfied},
  \end{cases}
\]
is a finite number.  Moreover, the $n$th eigenvalue $\lambda_n$, $n\in\NN$, $n\le N$,
satisfies
\begin{align}
\label{minmax1}
  \lambda_n &= \min_{\substack{L\subset\cD \\ \dim L =\kappa+n}}
  \;\;\sup_{x\in L\setminus\{0\}}\;\; p(x), \\[1ex]
\label{minmax2}
  \lambda_n &= \max_{\substack{L\subset H \\ \dim L = \kappa+n-1}}
  \;\; \inf_{\substack{x\in\cD\setminus\{0\} \\ x \perp L}}
  \; p(x).
\end{align}
For  subspaces $L$ with dimensions not considered in
\eqref{minmax1} and \eqref{minmax2} the right-hand side of
\eqref{minmax1} and \eqref{minmax2} gives values with the
following properties: if $\kappa>0$, then
\begin{equation}\label{minmax_n_le_kappa}
  \begin{aligned}
    \inf_{\substack{L\subset\cD \\ \dim L = n}}
    \;\; \sup_{x\in L\setminus\{0\}}\;\; p(x)
    \,&\le\, a \\[1ex]
    \sup_{\substack{L\subset H \\ \dim L = n-1}}
    \;\; \inf_{\substack{x\in\cD\setminus\{0\} \\ x \perp L}}\; p(x)
    \,&\le\, a
  \end{aligned}
  \qquad\text{for}\;n=1,\dots,\kappa;
\end{equation}
if $N<\infty$, then
\begin{equation}\label{minmax_n_g_N}
  \begin{aligned}
    \inf_{\substack{L\subset\cD \\ \dim L = n}}
    \;\;\sup_{x\in L\setminus\{0\}}\;\; p(x)
    \,&\ge\, \sup\Delta' \\[1ex]
    \sup_{\substack{L\subset H \\ \dim L = n-1}}
    \;\; \inf_{\substack{x\in\cD\setminus\{0\} \\ x \perp L}}\; p(x)
    \,&\ge\, \sup\Delta'
  \end{aligned}
  \qquad\text{for}\; n>\kappa+N \;\;\text{with}\; n\le\dim H.
\end{equation}
\end{theorem}

\begin{proof}
Let us first consider the case when (I), ($\searrow$) are satisfied.
We apply \cite[Theorem~2.1]{EL}.  Since $T$ is a holomorphic family of type (B),
\cite[Proposition~2.13]{EL} implies that conditions (i) and (ii) of \cite[Theorem~2.1]{EL}
are satisfied.  It follows directly from (I) and $(\searrow)$ that (iii) and (iv)
of \cite[Theorem~2.1]{EL} are also satisfied.  Now \cite[Theorem~2.1]{EL} implies that
$\sigma(T)\cap\Delta'$ is either empty or consists of a sequence of isolated eigenvalues
that can accumulate at most at $\sup\Delta'$.

Set
\[
  \Delta_1 \defeq \begin{cases}
    \Delta' & \text{if }N=0, \\[1ex]
    \bigl\{\mu\in\Delta'\mid\mu\le\lambda_1\bigr\} & \text{otherwise.}
  \end{cases}
\]
In \cite[Theorem~2.1]{EL} the number $\kappa$ was defined as
$\dim\cL_{(-\infty,0)}\bigl(T(a'')\bigr)$
with a particular choice of $a''\in\Delta_1$.  However, the function
\[
  \lambda \mapsto \dim\cL_{(-\infty,0)}\bigl(T(\lambda)\bigr)
\]
is constant on $\Delta_1$ by \cite[Lemma~2.6]{EL}.  Hence we choose
an arbitrary $a'\in\Delta_1$ for the definition of $\kappa$,
which by \cite[Theorem~2.1 and Lemma 2.6]{EL} is a finite number:
\[
  \kappa = \dim\cL_{(-\infty,0)}\bigl(T(a')\bigr).
\]
Let us now prove \eqref{minmax1}.
In \cite{EL} a special choice of a generalized Rayleigh functional
was considered; see Remark \ref{Welshpool}.
In order to distinguish it, we denote it by $q$, i.e.\
for $x\in\cD\setminus\{0\}$ we set
\[
  q(x) \defeq \begin{cases}
    \lambda_0 & \text{if } \frt(\lambda_0)[x] = 0, \\[0.5ex]
    -\infty & \text{if } \frt(\lambda)[x]<0 \text{ for all } \lambda\in\Delta, \\[0.5ex]
    +\infty & \text{if } \frt(\lambda)[x]>0 \text{ for all } \lambda\in\Delta.
  \end{cases}
\]
If $p(x)\in\Delta$ or $q(x)\in\Delta$ holds
for some  $x\in\cD\setminus\{0\}$, then by the definition of
$p$ and $q$ we have $ \frt(p(x))[x] = 0$ or $ \frt(q(x))[x] = 0$, respectively,
and thus $p(x)=q(x)$ follows.
In \cite[Theorem~2.1]{EL} it was proved that
\[
  \lambda_n = \min_{\substack{L\subset\cD \\ \dim L =\kappa+n}}
  \;\;\max_{x\in L\setminus\{0\}}\;\; q(x)
\]
for $n\in\NN$, $n\le N$.
Let $n\in\NN$ with $n\le N$.
There exists a subspace $L_0\subset\cD$ with $\dim L_0=\kappa+n$ such that
\[
  \max_{x\in L_0\setminus\{0\}} q(x) = \lambda_n,
\]
which implies in particular that $q(x) \le \lambda_n$ for all $x\in L_0\setminus\{0\}$.
If, for $x\in L_0\setminus\{0\}$, we have $q(x)=-\infty$, then $p(x)\leq a$ by the
definitions of $p$ and $q$, and hence $p(x)\le\lambda_n$.  If, for $x\in L_0\setminus\{0\}$,
we have $q(x)\ne-\infty$, then $q(x)\in\Delta$ and hence $p(x)=q(x)\le\lambda_n$.
This implies that
\begin{equation}\label{Shrewsbury}
  \sup_{x\in L_0\setminus\{0\}} p(x) \le \max_{x\in L_0\setminus\{0\}} q(x) = \lambda_n .
\end{equation}
Let $L\subset\cD$ be an arbitrary subspace with $\dim L=\kappa+n$.
Then, by the definition of $L_0$,
\[
  \max_{x\in L\setminus\{0\}} q(x) \ge \max_{x\in L_0\setminus\{0\}} q(x)= \lambda_n .
\]
Hence there exists an $x_0\in L\setminus\{0\}$
with $q(x_0)\ge\lambda_n$.  If $q(x_0)=+\infty$, then $p(x_0)\geq b$ and,
in particular, $p(x_0)\ge\lambda_n$.  If $q(x_0)\ne+\infty$, then $q(x_0)\in\Delta$,
which implies that $p(x_0)=q(x_0)\ge\lambda_n$.  Hence
\begin{equation}\label{ShrewsburyII}
  \sup_{x\in L\setminus\{0\}} p(x)\ge\lambda_n .
\end{equation}
By \eqref{Shrewsbury} and \eqref{ShrewsburyII} we obtain \eqref{minmax1}.
Equation \eqref{minmax2} is shown in a similar way.

Next we prove the first inequality in \eqref{minmax_n_le_kappa}.
Let $n\le\kappa$ and let $\lambda\in\Delta_1$ be arbitrary.
We have seen above that $\dim\cL_{(-\infty,0)}(T(\lambda))=\kappa$.
Therefore we can choose an $n$-dimensional subspace of $\cL_{(-\infty,0)}(T(\lambda))$,
which we denote by $L_0$ and which is contained in $\cD(T(\lambda))\subset\cD$.
Since $\frt(\lambda)[x]<0$ for all $x\in L_0\setminus\{0\}$, we have
\[
  \inf_{\substack{L\subset\cD \\ \dim L=n}} \;\; \sup_{x\in L\setminus\{0\}}\;\; p(x)
  \le \sup_{x\in L_0\setminus\{0\}}\;\; p(x)
  \le \lambda.
\]
This implies the first inequality in \eqref{minmax_n_le_kappa}
since $\lambda\in\Delta_1$ was arbitrary.
The second inequality in \eqref{minmax_n_le_kappa} is shown in a similar way.

We show the first inequality in \eqref{minmax_n_g_N}.
Let $n>\kappa+N$.
If  we have $\lambda_N =b=\sup \Delta'$, then \eqref{minmax_n_g_N} follows from \eqref{minmax1}. In all other cases, choose
$\lambda\in\Delta'$ such that $\lambda>\lambda_N$ if $N>0$.
It follows from \cite[Lemmas~2.6 and 2.7]{EL} that
$\dim\cL_{(-\infty,0)}(T(\lambda))=\kappa+N$.  Hence, for each subspace $L\subset\cD$
with $\dim L=n$, there exists an $x_0\in L\setminus\{0\}$ such that $\frt(\lambda)[x_0]\ge0$.
Therefore
\[
  \sup_{x\in L\setminus\{0\}}\; p(x) \ge p(x_0) \ge \lambda.
\]
Since this is true for every such $L$, we have
\[
  \inf_{\substack{L\subset\cD \\ \dim L=n}} \;\;
  \sup_{x\in L\setminus\{0\}} p(x) \ge \lambda,
\]
which implies the validity of the first inequality in \eqref{minmax_n_g_N}
as $\lambda$ can be chosen arbitrarily close to $\sup\Delta'$;
see \cite[Lemma 2.6]{EL}.
In a similar way one can show the second inequality in \eqref{minmax_n_g_N}.

If instead of (I), ($\searrow$) the assumptions (II), ($\nearrow$) are satisfied, then
the function $\wt T(\lambda)\defeq-T(\lambda)$ satisfies the assumptions (I), ($\searrow$)
and $\wt p(x)\defeq p(x)$  is a generalized Rayleigh functional for $\wt T$ on $\Delta$,
see Definition \ref{StJohn}.
Hence we can apply the already proved statements to $\wt T$, which
imply all assertions also
in this situation  as $\sigma_{\rm p}(\wt T) = \sigma_{\rm p}(T)$.
\end{proof}

\begin{remark}\label{rem_minmax}
\rule{0ex}{1ex}
\begin{enumerate}
\item[(i)]
Instead of assuming that $T$ is a holomorphic family of type {\rm(B)} it is
sufficient to assume some weaker continuity properties.  Also the domain of the quadratic
form may depend on $\lambda$.  For further details see \cite{EL}, in particular,
the assumptions (i) and (ii) there.
\item[(ii)]
If the functional $p$ is chosen such that it is continuous as a mapping from
$\cD$ into the extended real numbers $\RR\cup\{\pm\infty\}$ and $p(cx)=p(x)$
for all $c\in\CC\setminus\{0\}$ and $x\in\cD$,
then the supremum in \eqref{minmax1} is actually a maximum, i.e.\
the eigenvalue $\lambda_n$, $n\in \mathbb N$, $n\leq N$, satisfies
\[
  \lambda_n = \min_{\substack{L\subset\cD \\ \dim L =\kappa+n}}
  \;\;\max_{x\in L\setminus\{0\}}\;\; p(x).
\]
This follows from the fact that it is sufficient to take the supremum over the
set $\{x\in L\mid \|x\|=1\}$, which is compact. The same statement
applies to \eqref{minmax_n_le_kappa}
 and \eqref{minmax_n_g_N}.
\end{enumerate}
\end{remark}

\noindent
A similar theorem holds if we replace in Theorem \ref{th:var_left} the assumption
{\rm (I), ($\searrow$)}  by {\rm (I), ($\nearrow$)} and {\rm(II), ($\nearrow$)}
by {\rm(II), ($\searrow$)}, respectively, and change $\Delta'$ accordingly.
This is done in the following theorem.

\begin{theorem}\label{th:var_right}
Let $\Delta$ and $\Omega$ be as above.  Moreover, let $T(\lambda)$, $\lambda\in\Omega$,
be a family of closed operators in a Hilbert space $H$ satisfying either
{\rm (I), ($\nearrow$)} or {\rm(II), ($\searrow$)}, let $p$ be a generalized
Rayleigh functional and assume that
\[
  \Delta'\defeq \begin{cases}
  \Delta & \textit{if\, } \sess(T)\cap\Delta=\emptyset, \\[1ex]
  \bigl\{\lambda\in\Delta \mid \lambda>\sup\bigl(\sess(T)\cap\Delta\bigr)\bigr\}
  & \text{if\, } \sess(T)\cap\Delta\ne\emptyset,
  \end{cases}
\]
is non-empty.

Then $\sigma(T)\cap\Delta'$ is either empty or consists only of a finite or infinite
sequence of isolated eigenvalues of $T$ with finite geometric multiplicities, which
in the case of infinitely many eigenvalues in $\sigma(T)\cap\Delta'$ accumulates only
at $\inf\Delta'$ {\rm(}which equals $\sup(\sess(T)\cap\Delta)$
if $\sess(T)\cap\Delta\ne\emptyset$ and equals $a$ otherwise{\rm)}.

If $\sigma(T)\cap\Delta'$ is empty, then set $N\defeq0$; otherwise, denote the
eigenvalues in $\sigma(T)\cap\Delta'$ by $(\lambda_j)_{j=1}^N$, $N\in\NN\cup\{\infty\}$,
in non-increasing order, counted according to their geometric multiplicities:
$\lambda_1 \ge \lambda_2 \ge \cdots$.
Choose $b'\in\Delta'$ so that in the case $N>0$ it satisfies $\lambda_1 \le b'$.
Then the quantity
\[
  \kappa \defeq \begin{cases}
    \dim\mathcal L_{(-\infty,0)}\bigl(T(b')\bigr) & \text{if\, {\rm (I), ($\nearrow$)} are satisfied}, \\[1ex]
    \dim\mathcal L_{(0,\infty)}\bigl(T(b')\bigr) & \text{if\, {\rm (II), ($\searrow$)} are satisfied},
  \end{cases}
\]
is a finite number.  Moreover, the $n$th eigenvalue $\lambda_n$, $n\in\NN$, $n\le N$,
satisfies
\begin{align}
\label{minmax_re1}
  \lambda_n &= \max_{\substack{L\subset\cD \\ \dim L =\kappa+n}}
  \;\;\inf_{x\in L\setminus\{0\}}\;\; p(x), \\[1ex]
\label{minmax_re2}
  \lambda_n &= \min_{\substack{L\subset H \\ \dim L = \kappa+n-1}}
  \;\; \sup_{\substack{x\in\cD\setminus\{0\} \\ x \perp L}}
  \; p(x).
\end{align}
For  subspaces $L$ with dimensions not considered in
\eqref{minmax_re1} and \eqref{minmax_re2}  the right-hand side of
\eqref{minmax_re1} and \eqref{minmax_re2}  gives values with the
following properties: if $\kappa>0$, then
\begin{equation}\label{minmax_re_n_le_kappa}
  \begin{aligned}
    \sup_{\substack{L\subset\cD \\ \dim L = n}}
    \;\; \inf_{x\in L\setminus\{0\}}\;\; p(x)
    \,&\ge\, b \\[1ex]
    \inf_{\substack{L\subset H \\ \dim L = n-1}}
    \;\; \sup_{\substack{x\in\cD\setminus\{0\} \\ x \perp L}}\; p(x)
    \,&\ge\, b
  \end{aligned}
  \qquad\text{for}\;n=1,\dots,\kappa;
\end{equation}
if $N<\infty$, then
\begin{equation}\label{minmax_re_n_g_N}
  \begin{aligned}
    \sup_{\substack{L\subset\cD \\ \dim L = n}}
    \;\;\inf_{x\in L\setminus\{0\}}\;\; p(x)
    \,&\le\, \inf\Delta' \\[1ex]
    \inf_{\substack{L\subset H \\ \dim L = n-1}}
    \;\; \sup_{\substack{x\in\cD\setminus\{0\} \\ x \perp L}}\; p(x)
    \,&\le\, \inf\Delta'
  \end{aligned}
  \qquad\text{for}\; n>\kappa+N \;\;\text{with}\; n\le\dim H.
\end{equation}
\end{theorem}

\begin{proof}
The theorem follows from Theorem~\ref{th:var_left} applied to the function
$\wh T(\lambda)\defeq T(-\lambda)$, $-\lambda \in \Omega$.  With $\wh a\defeq -b$,
$\wh b\defeq -a$ and $\wh \Delta \defeq \left\{ -\lambda \mid \lambda \in \Delta\right\}$
all assumptions of Theorem~\ref{th:var_left} are satisfied, namely (I) and (II)
remain the same and $(\searrow)$ turns into $(\nearrow)$ and vice versa. That is,
$\wh T$ satisfies either {\rm (I), ($\searrow$)} or {\rm(II), ($\nearrow$)}.
Then the mapping $\wh p(x) \defeq -p(x)$ is a generalized Rayleigh functional
for $\wh T$ on $\wh \Delta$; see Definition~\ref{StJohn}.
Since $\wh\lambda_n=-\lambda_n$ for $\wh\lambda_n \in  \sigma_{\rm p}(\wh T)$,
all assertions of Theorem~\ref{th:var_right} follow from Theorem~\ref{th:var_left}.
\end{proof}

\begin{remark}\label{rem_minmax2}
If the functional $p$ is chosen such that it is continuous and $p(cx)=p(x)$
for  $c\in\CC\setminus\{0\}$ and $x\in\cD$ (see Remark
\ref{rem_minmax}),
then the infimum in \eqref{minmax_re1} is actually a minimum, i.e.\
the eigenvalue $\lambda_n$, $n\in \mathbb N$, $n\leq N$, satisfies
\begin{equation*}
\lambda_n = \max_{\substack{L\subset\cD \\ \dim L =\kappa+n}}
  \;\;\min_{x\in L\setminus\{0\}}\;\; p(x).
\end{equation*}
A similar statement
applies to \eqref{minmax_re_n_le_kappa}
 and \eqref{minmax_re_n_g_N}.
\end{remark}

\section{Framework}
\label{sec:framework}

Let $H$ be a Hilbert space and let $\fra_0$ and $\frd$ be
sesquilinear forms on $H$ with domains $\cD(\fra_0)$ and $\cD(\frd)$, respectively,
such that the following conditions are satisfied.
\myassumptionbf{F1}{
The sesquilinear form $\fra_0$ is densely defined, closed, symmetric and
bounded from below by a positive constant, i.e.\ $\exists\,c_1>0$ such that
$\fra_0[x]\ge c_1\|x\|^2$ for $x\in\cD(\fra_0)$.
}
\myassumptionbf{F2}{
The sesquilinear form $\frd$ is symmetric, satisfies $\cD(\frd)\supset\cD(\fra_0)$,
and there exists a $c_2>0$ such that
\[
  0 \le \frd[x] \le c_2\fra_0[x] \qquad\text{for all}\;\;x\in\cD(\fra_0).
\]
}

It is our aim to study the following second order differential equation
\begin{equation}\label{sys}
  \langle\ddot{z}(t),y \rangle + \frd[\dot{z} (t), y] + \fra_0 [z(t),y] = 0
  \qquad \text{for all } y\in \cD(\fra_0).
\end{equation}
In a first step we find an equivalent
Cauchy problem. Then, using the standard theory of semigroups,
we obtain  solutions of \eqref{sys}. Therefore we associate
with the form $\fra_0$ a positive definite self-adjoint operator $A_0$
with $\cD(A_0)\subset\cD(\fra_0)$ and $0\in\rho(A_0)$ via the First Representation Theorem \cite[Theorem~VI.2.1]{K}, i.e.\
\begin{equation}\label{reprA0}
  \fra_0[x,y] = \langle A_0x,y\rangle
  \qquad\text{for all}\;\;x\in\cD(A_0),\;y\in\cD(\fra_0).
\end{equation}
The operator $A_0$ is called \emph{stiffness operator}.  The Second Representation Theorem \cite[Theorem~VI.2.6]{K} shows $\cD(A_0^{1/2})=\cD(\fra_0)$ and
\[ \fra_0[x,y] = \langle A_0^{1/2}x,A_0^{1/2}y\rangle
  \qquad\text{for all}\;\;x,y\in \cD(\fra_0).\]
We define the two spaces
\begin{equation}\label{Cambridge11}
  H_{\frac{1}{2}} \defeq \dom (A_0^{1/2}) \qquad \text{with norm} \quad
  \|x\|_{H_{\frac{1}{2}}} \defeq \bigl\|A_0^{1/2} x\bigr\|_{H}
\end{equation}
and
\begin{equation}\label{Cambridge12}
  \begin{aligned}
    & H_{-\frac{1}{2}} \text{ as the completion of $H$ with respect to the norm} \\
    & \|x\|_{H_{-\frac{1}{2}}} \defeq \bigl\|A_0^{-1/2}x\bigr\|_H.
  \end{aligned}
\end{equation}
By continuity, $A_0$ and $A_0^{1/2}$ can be extended to isometric
isomorphisms from $H_{\frac{1}{2}}$ onto $H_{-\frac{1}{2}}$ and from $H$
onto $H_{-\frac{1}{2}}$, respectively.  These extensions are also denoted
by $A_0$ and $A_0^{1/2}$.  The space $H_{-\frac{1}{2}}$ can
be identified with the dual space of $H_{\frac{1}{2}}$ by identifying elements
$x\in H_{-\frac{1}{2}}$ with bounded linear functionals on $H_{\frac{1}{2}}$ as follows
\begin{equation}\label{duality_innprod}
  \langle x,y\rangle_{\mphhalf} \defeq \bigl\langle A_0^{-1/2}x,A_0^{1/2}y\bigr\rangle,
  \qquad x\in H_{-\frac{1}{2}}, \,y\in H_{\frac{1}{2}}.
\end{equation}
Note that, for $x\in H$, $y\in H_{\frac{1}{2}}$, we have
\begin{equation}\label{Wuppi1}
  \langle x,y\rangle_{H_{-\frac{1}{2}}\times H_{\frac{1}{2}}}=\langle x,y\rangle_H.
\end{equation}
The form $\fra_0$ can be expressed in terms of the extended operator $A_0$:
\begin{equation}\label{a0A0}
  \fra_0[x,y] = \langle A_0x,y\rangle_{H_{-\frac12}\times H_{\frac12}}
  \qquad\text{for all}\;\;x,y\in H_{\frac12};
\end{equation}
this relation is obtained from \eqref{reprA0} by continuous extension.

Assumption \textbf{(F2)} implies that $\frd$ restricted to $H_{\frac12}$
is a bounded, non-negative, symmetric sesquilinear form on the Hilbert space $H_{\frac12}$.
Hence, by \cite[Theorem VI.2.7]{K}
there exists a bounded, self-adjoint, non-negative operator $\wt D$
on $H_{\frac12}$ such that
\[
  \frd[x,y] = \bigl\langle\wt Dx,y\bigr\rangle_{H_{\frac12}}
  \qquad\text{for all}\;\;x,y\in H_{\frac12}.
\]
Now we define the \emph{damping operator} $D$ by
\[
  D \defeq A_0\wt D,
\]
where $A_0$ is considered as a bounded operator from $H_{\frac12}$
onto $H_{-\frac12}$.
Clearly, the operator $D$ is bounded from $H_{\frac12}$ to $H_{-\frac12}$.
Using \eqref{duality_innprod} we obtain the following connection
between $\frd$ and $D$:
\begin{equation}\label{reprD}
\begin{aligned}
  \frd[x,y] &= \bigl\langle \wt Dx,y\bigr\rangle_{H_{\frac12}}
  = \bigl\langle A_0^{1/2}\wt Dx,A_0^{1/2}y\bigr\rangle \\[0.5ex]
  &= \bigl\langle A_0^{-1/2}Dx,A_0^{1/2}y\bigr\rangle
  = \langle Dx,y\rangle_{H_{-\frac12}\times H_{\frac12}}
\end{aligned}
\end{equation}
for $x,y\in H_{\frac12}$.

We consider the following standard first-order evolution equation
\begin{equation}\label{diffeqA}
  \dot{x} (t) = {\mathcal A} x(t)
\end{equation}
in the space $\cH\defeq H_{\frac{1}{2}} \times H$
where $\cA: \dom(\cA)\subset\cH\rightarrow\cH$ is given by
\begin{align}\label{spaetabends}
  & \cA = \begin{bmatrix} 0\; & \;I \\[0.5ex] -A_0\; & \;-D \end{bmatrix}, \\[1ex]
  \label{spaetabends2}
  & \dom(\cA) = \left\{\vect{z}{w} \in H_{\frac{1}{2}} \times H_{\frac{1}{2}} \Bigm|
  A_0 z + D w \in H  \right\}.
\end{align}
It is easy to see (e.g.\ \cite{TWII}) that $\cA$ has a bounded inverse in $\cH$ given by
\begin{equation} \label{NaDann}
  \cA^{-1} = \begin{bmatrix} -A_0^{-1}D\; & \;-A_0^{-1} \\[1ex] I & 0 \end{bmatrix} = \begin{bmatrix} -\wt D\; & \;-A_0^{-1} \\[1ex] I & 0 \end{bmatrix},
\end{equation}
where $A_0^{-1}D$ is considered as an operator acting in $H_{\frac{1}{2}}$
and $I$ is the embedding from $H_{\frac12}$ into $H$.
The operator $\cA$ itself is not self-adjoint in the Hilbert space $\cH$.
However, with
\begin{equation*}
  J \defeq \begin{bmatrix} I\; & \;0 \\[0.3ex] 0\; & \;-I \end{bmatrix}
\end{equation*}
the operator $J\cA$ is symmetric in $\cH$.
Since $\cA$ has a bounded inverse, the operator $J\cA$ is even self-adjoint in $\cH$.
Therefore,
\begin{equation*}\label{defj}
  \cA^*= J\cA J,\qquad \text{ with } \dom(\cA^*)=J\dom(\cA)
\end{equation*}
(see also  \cite[Proof of Lemma~4.5]{WT}) and
\begin{equation*}
  \Re \langle \cA x,x \rangle \leq 0 \quad \text{for } x \in \dom(\cA) \quad
  \text{and} \quad
  \Re \langle \cA^* x,x \rangle \leq 0 \quad \text{for } x \in \dom(\cA^*).
\end{equation*}
This implies that $\cA$ is the generator of a strongly continuous semigroup
of contractions on the state space $\cH$.
This fact is well known; see, e.g.\ \cite{bi,biw,CLL,HS,las} or
\cite[Proposition~5.1]{WT}. Hence, \eqref{diffeqA} together
with an appropriate initial value has a unique (classical) solution.
This implies the following proposition.

\begin{proposition} \label{Nottingham}
Assume that {\rm\textbf{(F1)}--\textbf{(F2)}} are satisfied. For
$z_0, w_0 \in  H_{\frac{1}{2}}$ with $A_0z_0 + Dw_0 \in H$ there exists
a solution $z:\mathbb R^+\to H_{\frac{1}{2}}$ of \eqref{sys} that satisfies
\begin{itemize}
\item $z(0)=z_0$ and $\dot{z}(0) = w_0$;
\item the function $z$ is continuously differentiable in $H_{\frac{1}{2}}$;
\item the function $\dot{z}$ is continuously differentiable in $H$.
\end{itemize}
Moreover, a solution of \eqref{sys} with the above properties is unique
and equals the first component of the classical solution of the Cauchy
problem
\begin{equation}\label{RobinHood}
  \vect{\dot{z}}{\dot{w}}= \cA \vect{z}{w}, \qquad
  \vect{z(0)}{w(0)} = \vect{z_0}{w_0}
\end{equation}
with $\binom{z_0}{w_0}\in\dom(\cA)$.
\end{proposition}

\medskip

We mention that a similar relation holds for mild solutions of the Cauchy problem
\eqref{RobinHood} with $\binom{z_0}{w_0}$ in $\cH$ instead of $\dom(\cA)$
and a somehow weaker formulation
of \eqref{sys},
\begin{equation}\label{weakequ}
  \frac{\rd}{\rd t}\Bigl(\langle\dot{z}(t),y\rangle + \frd[z(t),y]\Bigr) + \fra_0[z(t),y] = 0
  \qquad \text{for all } y\in \cD(\fra_0).
\end{equation}
For details we refer to \cite[Theorem~2.2]{CLL}, see also \cite{biw}.

\begin{remark}
The operators $A_0$ and $D$ satisfy the following conditions
\textbf{(A1)} and \textbf{(A2)}, which appeared in various papers;
see, e.g.\ \cite{jatr,JTW,JT09}.
\myassumptionbf{A1}{
The stiffness operator $A_0 : \dom (A_0) \subset H \rightarrow H$
is a self-adjoint, positive definite linear operator on a Hilbert
space $H$ such that $0\in\rho(A_0)$.
}
\myassumptionbf{A2}{
The damping operator $D: H_{\frac{1}{2}} \rightarrow
H_{-\frac{1}{2}}$ is a bounded operator with
\[
  \langle Dz, z\rangle_{H_{-\frac{1}{2}}\times
  H_{\frac{1}{2}}} \ge 0  ,\qquad z\in H_{\frac{1}{2}}.
\]
}
\vspace*{-1ex}
Instead of starting with the forms and then constructing the operators
one could also start with two operators $A_0$ and $D$ that
satisfy \textbf{(A1)} and \textbf{(A2)} and then define the
sesquilinear forms $\fra_0$ and $\frd$ via
\begin{equation*}
  \begin{aligned}
    \fra_0[x,y] \defequ \langle A_0x,y\rangle_{H_{-\frac{1}{2}}\times
    H_{\frac{1}{2}}}, \\[1ex]
    \frd[x,y] \defequ \langle Dx,y\rangle_{H_{-\frac{1}{2}}\times
    H_{\frac{1}{2}}},
  \end{aligned}
  \qquad x,y\in H_{\frac12}.
\end{equation*}
It is easy to see that these forms satisfy \textbf{(F1)} and \textbf{(F2)}.
\end{remark}

In the following we study the spectrum of $\mathcal A$.
For $\smvect{x_1}{y_1},\smvect{x_2}{y_2} \in  H_{\frac{1}{2}}\times H$ we define
an indefinite inner product on $\cH$ by
\begin{equation*}
  \left[ \vect{x_1}{y_1},\vect{x_2}{y_2}\right] \defeq
  \left\langle J \vect{x_1}{y_1},\vect{x_2}{y_2}\right\rangle
  = \langle x_{1},x_{2} \rangle_{H_\frac{1}{2}}-\langle y_{1},y_{2} \rangle.
\end{equation*}
Then $(\cH, \Skindef )$ is a Krein space and $\cA$ is a self-adjoint operator
with respect to $\Skindef$ (note that the latter is equivalent to the
self-adjointness of $J\cA$ in $\cH$). Hence $\sigma(\cA)$ is
symmetric with respect to $\mathbb R$; see, e.g.\
\cite[Theorem VI.6.1]{B}.
 For the basic theory of Krein spaces
and operators acting therein we refer to \cite{AI} and \cite{B}.
In the following proposition we collect the above considerations.

\begin{proposition}
\label{prop:cont}
If {\rm\textbf{(F1)}} and {\rm\textbf{(F2)}} are satisfied, then
the operator $\cA$ is self-adjoint in the Krein space
$(\cH, \Skindef )$, its spectrum is contained
in the closed left half-plane and is symmetric with respect to the real line.
The operator $\cA$ has a bounded inverse,
and it is the generator of a
strongly continuous semigroup of contractions on the state space $\cH$.
\end{proposition}


Proposition~\ref{prop:cont} guarantees that the spectrum
of ${\mathcal A}$ is contained in ${\mathbb C}_{-}$, where ${\mathbb C}_{-}$
denotes the closed left half-plane $\{z\in\CC \mid \Re z \leq 0\}$.
Since $\cA$ has a bounded inverse, we even have
$\sigma(\mathcal A)\subset\CC_-\setminus\{0\}$.
However, apart from this restriction and the symmetry with respect to the real line,
the spectrum of ${\mathcal A}$ is quite arbitrary; see, e.g.\ \cite[Examples~3.5 and 3.6]{JMT}
and we refer to Example~3.2 in \cite{jatr}.

For the rest of the paper we assume that, in addition to \textbf{(F1)} and \textbf{(F2)},
also the following condition is satisfied.

\myassumptionbf{F3}{
The operator $A_0^{-1}$ is a compact operator in $H$.
}

\medskip

\noindent
In the following we consider $\wt D=A_0^{-1}D$ and $A_0^{-1/2}DA_0^{-1/2}$
as bounded operators acting in $H_{\frac{1}{2}}$ and $H$, respectively.
For $\lambda \in \mathbb C$ the relations
\begin{align*}
  \ker\bigl(A_0^{-1/2}DA_0^{-1/2}-\lambda \bigr)
  &= A_0^{1/2}\Bigl(\ker\bigl(\wt D-\lambda \bigr)\Bigr), \\[0.5ex]
  \ran\bigl(A_0^{-1/2}DA_0^{-1/2}-\lambda \bigr)
  &= A_0^{1/2}\Bigl(\ran\bigl(\wt D-\lambda \bigr)\Bigr)
\end{align*}
hold.  This, together with the fact that $A_0^{1/2}$ is an isomorphism
from $H_{\frac{1}{2}}$ onto $H$, implies that
\begin{equation}\label{Wuppi2}
  \sigma\bigl(A_0^{-1/2}DA_0^{-1/2}\bigr) = \sigma\bigl(\wt D\bigr),
  \qquad
  \sess\bigl(A_0^{-1/2}DA_0^{-1/2}\bigr) = \sess\bigl(\wt D\bigr).
\end{equation}
In the next definition we introduce some numbers that are used
in the following proposition for a further description of the spectrum of $\cA$
and in the next section in connection with the study of a quadratic operator polynomial.

\begin{definition}
Set
\begin{equation}\label{Elbersfeld}
  \delta \defeq \min\sigma\bigl(A_0^{-1/2}DA_0^{-1/2}\bigr), \qquad
  \gamma \defeq \max\sigma\bigl(A_0^{-1/2}DA_0^{-1/2}\bigr).
\end{equation}
If $H$ is finite-dimensional, then 
set
\begin{equation}\label{Elbersfeldberg}
  \delta_0 \defeq +\infty, \qquad \gamma_0 \defeq 0;
\end{equation}
otherwise, set
\begin{equation}\label{Kuchen}
  \delta_0 \defeq \min\sess\bigl(A_0^{-1/2}DA_0^{-1/2}\bigr), \qquad
  \gamma_0 \defeq \max\sess\bigl(A_0^{-1/2}DA_0^{-1/2}\bigr).
\end{equation}
Moreover, if $H$ is infinite-dimensional, $\delta_0=0$ and $\gamma_0>0$, then set
\begin{equation}\label{defdelta1}
  \delta_1 \defeq \inf\bigl(\sess\bigl(A_0^{-1/2}DA_0^{-1/2}\bigr)\setminus\{0\}\bigr).
\end{equation}
\end{definition}

\medskip

If $H$ is infinite-dimensional, then clearly
$0\leq \delta\leq \delta_0\leq \gamma_0\leq \gamma$.
The numbers $\delta$ and $\gamma$ can be expressed in terms of the
forms $\fra_0$ and $\frd$:
\begin{equation}\label{deltainf}
\begin{aligned}
  \delta &= \inf_{x\in H\setminus\{0\}}
  \frac{\bigl\langle A_0^{-1/2}DA_0^{-1/2}x,x\bigr\rangle}{\|x\|^2}
  = \inf_{y\in H_{1/2}\setminus\{0\}}
  \frac{\langle Dy,y\rangle_{H_{-\frac{1}{2}}\times H_{\frac{1}{2}}}}{\langle A_0y,
  y\rangle_{H_{-\frac{1}{2}}\times H_{\frac{1}{2}}}}
  \\[1ex]
  &= \inf_{y\in H_{1/2}\setminus\{0\}} \frac{\frd[y]}{\fra_0[y]}\,,
\end{aligned}
\end{equation}
where we made the substitution $y=A_0^{-1/2}x$,
and similarly
\begin{equation}\label{gammasup}
  \gamma = \sup_{y\in H_{1/2}\setminus\{0\}} \frac{\frd[y]}{\fra_0[y]}\,.
\end{equation}
If $H$ is infinite-dimensional, then one can use the standard variational
principle for bounded operators to express $\delta_0$ and $\gamma_0$
in terms of $\fra_0$ and $\frd$:
\begin{equation}\label{delta0inf}
  \delta_0 = \sup_{n\in\NN}\;\; \inf_{\substack{L\subset H_{1/2} \\ \dim L=n}}\;\;
  \sup_{y\in L\setminus\{0\}}\;
  \frac{\frd[y]}{\fra_0[y]}\,, \qquad\qquad
  \gamma_0 = \inf_{n\in\NN}\;\; \sup_{\substack{L\subset H_{1/2} \\ \dim L=n}}\;\;
  \inf_{y\in L\setminus\{0\}}\;
  \frac{\frd[y]}{\fra_0[y]}\,.
\end{equation}

\begin{proposition} \label{Ilmenau}
Assume that {\rm\textbf{(F1)}--\textbf{(F3)}} are satisfied.
Then
\begin{align}
  \label{esssA}
  \sess(\mathcal A) &= \left\{\lambda\in\mathbb C\backslash\{0\}\,\Bigm|\,
    \frac{1}{\lambda}\in \sess\bigl(-\wt D\bigr)\right\} \\[1ex]
  \label{esssA2}
  &= \left\{\lambda\in\mathbb C\backslash\{0\}\,\Bigm|\,
    \frac{1}{\lambda}\in \sess\bigl(-A_0^{-1/2}DA_0^{-1/2}\bigr)\right\} \\[1ex]
  \label{esssA3}
  &\subset (-\infty,0).
\end{align}
The spectrum in $\CC\setminus\sess(\cA)$ is a discrete set consisting only of
eigenvalues.
Moreover, the set $\sigma(\cA)\setminus\RR$ has no finite accumulation point.

Moreover, the following statements are true:
\begin{itemize}
\item
if $\gamma_0=0$, then $\sess(\mathcal A)=\emptyset$;
\item
if $\gamma_0>0$ and $\delta_0=0$, then
\begin{align*}
  \inf\sess(\mathcal A) &= \begin{cases}
    -\infty & \text{if}\;\;\delta_1=0, \\[1ex]
    -\dfrac{1}{\delta_1} & \text{if}\;\;\delta_1>0,
  \end{cases} \\[2ex]
  \max\sess(\mathcal A) &= -\frac{1}{\gamma_0}\,;
\end{align*}
\item
if $\delta_0>0$, then
\[
  \min\sess(\mathcal A) = -\frac{1}{\delta_0} \qquad\text{and}\qquad
  \max\sess(\mathcal A) = -\frac{1}{\gamma_0}\,.
\]
\end{itemize}
\end{proposition}

\begin{proof}
The equality in \eqref{esssA} was proved in \cite[Theorem~4.1]{jatr}.
Relation \eqref{Wuppi2} implies \eqref{esssA2}, and \eqref{esssA3}
follows from assumption \textbf{(F2)}.
The discreteness of the spectrum in $\CC\setminus\sess(\cA)$ follows from Fredholm
theory and the fact that $\CC\setminus\sess(\cA)$ is a connected set and has non-empty
intersection with $\rho(\cA)$, namely $0\in\rho(\cA)\cap(\CC\setminus\sess(\cA))$
by \eqref{NaDann}.
Corollary~5.2 in \cite{jatr} implies that no point
from $\sess(\cA)$ is an accumulation point of the non-real spectrum of $\cA$,
which shows that the non-real spectrum has no finite accumulation point.
The remaining assertions are clear.
\end{proof}

Note that, although $A_0^{-1}$ is compact, the operator $\mathcal A^{-1}$
is in general not a compact operator in $\cH$.
In fact, $\mathcal A^{-1}$ is compact if and only if the operator $D$ is compact
as an operator acting from $H_{\frac{1}{2}}$ into $H_{-\frac{1}{2}}$;
see \cite[Lemma~3.2]{T}.

\section{A quadratic operator polynomial}
\label{sec:poly}

In the following we construct a quadratic operator polynomial $T(\lambda)$
that is connected with the operator $\mathcal A$ and also the differential
equation \eqref{sys}.
Throughout this section let $\fra_0$ and $\frd$ be sesquilinear forms that
satisfy \textbf{(F1)}--\textbf{(F3)} from Section~\ref{sec:framework}.
Moreover, let the operators $A_0$, $D$, $\cA$ and the numbers
$\delta$, $\gamma$, $\delta_0$, $\gamma_0$ be as in Section~\ref{sec:framework}.
It follows from \eqref{deltainf} and \eqref{gammasup} that
\begin{equation}\label{inequ_a0_d}
  \delta\fra_0[x] \le \frd[x] \le \gamma\fra_0[x], \qquad x\in H_{\frac{1}{2}}.
\end{equation}

Before we define the operator polynomial $T(\lambda)$,
we need two lemmas.

\begin{lemma}\label{lemeps}
Let $R$ be a compact  operator in $H$ and $\eps$ an arbitrary positive number.
Then there exists a constant $C\ge0$ such that
\[
  \|RA_0^{1/2}x\|^2 \le \eps\|A_0^{1/2}x\|^2 + C\|x\|^2
  \qquad\text{for all }x\in H_{\frac12}.
\]
\end{lemma}

\begin{proof}
The operator $RA_0^{1/2}A_0^{-1/2}=R$ is a compact operator in $H$.
Hence $RA_0^{1/2}$ is $A_0^{1/2}$-compact; see, e.g.\ \cite[Section~IV.1.3]{K}.
By \cite[Corollary III.7.7]{EE},  $RA_0^{1/2}$ has $A_0^{1/2}$-bound $0$,
which implies the assertion (see \cite[\S V.4.1]{K}).
\end{proof}

Define the following set, on which the operator polynomial $T(\lambda)$ will be defined:
\begin{equation}\label{defPhigamma0}
  \Phi_{\gamma_0} \defeq \begin{cases}
    \biggl\{z\in\mathbb{C} \Bigm| |z|<\dfrac{1}{\gamma_0}\biggr\} & \text{if } \gamma_0 \ne 0, \\[2ex]
    \mathbb C & \text{if } \gamma_0 = 0.
  \end{cases}
\end{equation}

\begin{lemma}\label{mnx}
For $\lambda\in\Phi_{\gamma_0}$ the form $\lambda\frd$ is
relatively bounded with respect to $\fra_0$ with $\fra_0$-bound less than $1$,
i.e.\ there exist real constants $C_1,C_2$ with $C_1\ge0$, $0\le C_2<1$ such that
\begin{equation*}
  \bigl|\lambda\frd[x]\bigr| \le C_1\|x\|^2 + C_2\fra_0[x] \qquad
  \text{for all } x\in H_{\frac12}=\cD(\fra_0).
\end{equation*}
\end{lemma}

\begin{proof}
Obviously, for $\lambda=0$ the assertion of Lemma \ref{mnx} is true.
Let $\lambda\in\Phi_{\gamma_0}\setminus\{0\}$ and choose $\gamma'\in\RR$ such that
$\gamma_0<\gamma'<\frac{1}{|\lambda|}$.
Denote by $E$ the spectral function in $H$ corresponding to the bounded self-adjoint
operator $S \defeq A_0^{-1/2}DA_0^{-1/2}$.
Then, for $x\in H_{\frac12}$, we have
\begin{align*}
  \bigl|\frd[x]\bigr|
  &= \langle Dx,x\rangle_{H_{-\frac12}\times H_{\frac12}}
  = \langle A_0^{-1/2}Dx,A_0^{1/2}x\rangle
  = \langle S A_0^{1/2}x,A_0^{1/2}x\rangle \\[1ex]
  &= \bigl\langle S E([0,\gamma'])A_0^{1/2}x,E([0,\gamma'])A_0^{1/2}x\bigr\rangle \\[1ex]
    &\quad + \bigl\langle S E((\gamma',\infty))A_0^{1/2}x,E((\gamma',\infty))A_0^{1/2}x\bigr\rangle \\[1ex]
  &\le \gamma'\bigl\|E([0,\gamma'])A_0^{1/2}x\bigr\|^2
    + \bigl\|S^{1/2}E((\gamma',\infty))A_0^{1/2}x\bigr\|^2 \\[1ex]
  &\le \gamma'\|A_0^{1/2}x\|^2
    + \bigl\|S^{1/2}E((\gamma',\infty))A_0^{1/2}x\bigr\|^2.
\end{align*}
By the definition of $\gamma_0$ and the fact that $\gamma'>\gamma_0$ it follows
that $E((\gamma',\infty))$ is a finite rank projection.
Choose $\eps>0$ such that $|\lambda|(\gamma'+\eps)<1$, which is possible
because $\gamma'<\frac{1}{|\lambda|}$.
Then Lemma~\ref{lemeps} applied to the finite rank operator $S^{1/2}E((\gamma',\infty))$
implies that there exists a $C\ge0$ such that
\begin{align*}
  \bigl|\lambda\frd[x]\bigr|
  &\le |\lambda|\gamma'\|A_0^{1/2}x\|^2
  + |\lambda|\Bigl(\eps\|A_0^{1/2}x\|^2 + C\|x\|^2\Bigr) \\[1ex]
  &= |\lambda|(\gamma'+\eps)\fra_0[x] + |\lambda|C\|x\|^2,
\end{align*}
which shows that $\lambda\frd$ is $\fra_0$-bounded with $\fra_0$-bound less than $1$.
\end{proof}

For $\lambda\in\CC$ we define the sesquilinear form $\frt(\lambda)$ with
domain $\cD(\frt(\lambda))=H_{\frac{1}{2}}$ by
\begin{equation}\label{Wuppi5}
  \frt(\lambda)[x,y] \defeq \lambda^2\langle x,y\rangle + \lambda\frd[x,y] + \fra_0[x,y]
  \qquad x,y\in H_{\frac12},
\end{equation}
and the corresponding quadratic form by $\frt(\lambda)[x] \defeq \frt(\lambda)[x,x]$
for $x\in H_{\frac12}$.
Note that if a function of the form $z(t)=e^{\lambda t}x$ with $x\in H_{\frac12}$
is plugged into \eqref{sys}, then one obtains the equation $\frt(\lambda)[x,y]=0$.
Using \eqref{a0A0} and \eqref{reprD} we can rewrite $\frt(\lambda)$ as follows:
\begin{equation}
  \frt(\lambda)[x,y] \defeq \bigl\langle \lambda^2 x+\lambda Dx+A_0x,y
  \bigr\rangle_{H_{-\frac12}\times H_{\frac12}}
  \qquad x,y\in H_{\frac12}.
\end{equation}
In the next proposition we introduce the representing operators $T(\lambda)$
for $\lambda\in\Phi_{\gamma_0}$ and state some of their properties.

\begin{proposition}\label{HotelCentral}
For $\lambda\in\Phi_{\gamma_0}$ the form $\frt(\lambda)$ with
domain $\cD(\frt(\lambda))=H_{\frac{1}{2}}$
is a closed sectorial form in $H$.  The m-sectorial operator $T(\lambda)$ in
$H$  that is
associated with $\frt(\lambda)$ is given by
\begin{align*}
  \dom(T(\lambda)) &=  \left\{ x \in H_{\frac12} \mid \lambda Dx + A_0x
  \in H \right\},
  \\[1ex]
  T(\lambda)x &= \lambda^2 x + \lambda D x + A_0x, \hspace*{15ex} x\in\dom(T(\lambda)).
\end{align*}
The family $T(\lambda)$, $\lambda \in \Phi_{\gamma_0}$, of m-sectorial operators
is a holomorphic family of type {\rm(B)}, which satisfies $T(\overline\lambda)=T(\lambda)^*$
for $\lambda\in\Phi_{\gamma_0}$.
For $\lambda\in\Phi_{\gamma_0}\cap\RR$ the operators $T(\lambda)$ are self-adjoint
and bounded from below.
\end{proposition}

\begin{proof}
Since $\fra_0$ is a closed symmetric non-negative form and, by Lemma~\ref{mnx},
$\lambda\frd$ is bounded with respect to $\fra_0$ with $\fra_0$-bound less than 1,
it follows from \cite[Theorem~VI.1.33]{K} that $\frt(\lambda)$ is closed and sectorial
for $\lambda \in \Phi_{\gamma_0}$.
Hence by \cite[Theorem~VI.2.1]{K} there exist m-sectorial operators $T(\lambda)$
that represent the forms $\frt(\lambda)$.
The form of the domain and the action of $T(\lambda)$ follow easily
from \cite[Theorem~VI.2.1]{K}.
The domain of $\frt(\lambda)$ is independent of $\lambda$, and the analyticity
of $\lambda\mapsto\frt(\lambda)[x]$ is clear.  Hence $T$ is a holomorphic family of type (B).
Since $\frt(\overline\lambda)[x,y]=\overline{\frt(\lambda)[y,x]}$,
we have $T(\overline\lambda)=T(\lambda)^*$; see \cite[Theorem~VI.2.5]{K}.
From this we obtain also the self-adjointness of $T(\lambda)$
for $\lambda\in\Phi_{\gamma_0}\cap\RR$;
moreover, $T(\lambda)$ is bounded from below in this case since it is m-sectorial.
\end{proof}

Next we show that on $\Phi_{\gamma_0}$ the spectral problems for $\mathcal A$ and $T$
are equivalent.

\begin{proposition}\label{propdiskrho}
Consider $T$ as a function defined on $\Phi_{\gamma_0}$.
On $\Phi_{\gamma_0}$ the spectra and point spectra of $\mathcal A$ and $T$ coincide, i.e.\
\begin{equation}\label{BenianCourt}
  \sigma_{\rm p}(\mathcal A) \cap  \Phi_{\gamma_0} =
  \sigma(\mathcal A) \cap  \Phi_{\gamma_0} = \sigma(T)=
  \sigma_{\rm p}(T).
\end{equation}
For  $\lambda_0\in \sigma_{\rm p}(\mathcal A) \cap \Phi_{\gamma_0}$
the geometric multiplicities coincide:
\begin{equation}\label{dimker}
  \dim \ker (\mathcal A - \lambda_0) = \dim \ker T(\lambda_0).
\end{equation}
Moreover,
\[
  \sess(T) =\emptyset.
\]
If $\gamma_0 \ne 0$, then there are at most finitely many eigenvalues
of $\mathcal A$ {\rm(}and, hence, of $T${\rm)} in $\Phi_{\gamma_0} \setminus \RR$.
\end{proposition}

\begin{proof}
First we show equality of the point spectra of $\mathcal A$ and $T$. For this,
let $\lambda\in\Phi_{\gamma_0}$ and assume that
$0\in \sigma_p(T(\lambda))$. Then there exists $x\in \dom (T(\lambda))\setminus\{0\}$
with $\lambda^2x+\lambda Dx +A_0x=0$.
Therefore
$\smvect{x}{\lambda x} \in \dom (\mathcal A)$ and
\[
  (\mathcal A -\lambda)\vect{x}{\lambda x} = 0.
\]
Conversely, if $\lambda \in \sigma_p(\mathcal A)$ and if
$\smvect{x}{y} \in \dom (\mathcal A)$
is a corresponding eigenvector, one concludes that
\begin{equation}\label{Steingasse}
  y=\lambda x \quad \text{and} \quad A_0x +  Dy + \lambda y=0.
\end{equation}
Hence $x\in \dom (T(\lambda))$ and $T(\lambda)x=0$ with $x\ne0$ because otherwise,
$\smvect{x}{y}=0$.
Therefore the point spectra of $\mathcal A$ and $T$ coincide in $\Phi_{\gamma_0}$.
Moreover, as the first component of an eigenvector
$\smvect{x}{\lambda x} \in \dom (\mathcal A)$
of $\mathcal A$ satisfies  $x \in \dom (T(\lambda))$ and $T(\lambda)x=0$ and vice versa,
the statement on the geometric multiplicities follows.

Next assume that $\lambda\in\rho(\mathcal A) \cap \Phi_{\gamma_0}$. Then for
$g\in H$ there exists $\smvect{x}{y} \in \dom (\mathcal A)$ with
\[
  (\mathcal A -\lambda)\vect{x}{y} = \vect{0}{g}.
\]
From this one concludes that
\[
  y=\lambda x \quad \text{and} \quad A_0x +  Dy + \lambda y=g,
\]
which shows that $x\in  \dom (T(\lambda))$ and $T(\lambda)x=g$.
Hence $T(\lambda)$ is surjective and, by the already proved statement about
the eigenvalues, $\lambda\in\rho(T)$.
Proposition~\ref{Ilmenau} implies that
$\sess(\mathcal A) \cap \Phi_{\gamma_0} = \emptyset$ which, together with
$0\in\rho(\mathcal A)$ (see Proposition~\ref{prop:cont}), gives the first
equality in \eqref{BenianCourt}. Hence each point $\lambda$ in $\Phi_{\gamma_0}$
is either an eigenvalue of $\mathcal A$ and, hence, of $T$, or belongs to
the resolvent set of $\mathcal A$ and hence of $T$.
This proves \eqref{BenianCourt}.

We show the statement about the essential spectrum of $T$.
Let $\lambda\in\Phi_{\gamma_0}$.
The statement is obvious for finite-dimensional $H$; hence let
$H$ be infinite-dimensional.
By Lemma~\ref{mnx} there exist constants $a,b$ such that $a\ge0$, $0\le b<1$ and
\[
  \bigl|\lambda \frd [x] \bigr|
  \le a\|x\|^2 + b\fra_0 [x], \quad x\in
  H_{\frac{1}{2}}.
\]
Denote by $L$ the spectral subspace for $A_0$ corresponding to the
interval $\left[ 0, \frac{|\lambda|^2 +a}{1-b} +1\right]$.
Assume that $0\in\sess(T(\lambda))$.
It follows from Proposition~\ref{Ilmenau} and the definition of $\Phi_{\gamma_0}$
that $\lambda\notin\sess(\cA)$ and $\overline\lambda\notin\sess(\cA)$.
Hence \eqref{dimker} implies that $\dim\ker T(\lambda)<\infty$ and
\[
  \dim\bigl(\ran T(\lambda)\bigr)^\perp = \dim\bigl(\ker T(\lambda)^*\bigr)
  = \dim\ker T(\overline\lambda) < \infty.
\]
By \cite[Theorem~IX.1.3]{EE} there exists a singular
sequence $(x_n)_{n\in\mathbb N}$ with $x_n\in\mathcal D(T(\lambda))$,
$\|x_n\|=1$, $x_n \rightharpoonup 0$ (i.e.\ $x_n$ converges to $0$ weakly) and
$T(\lambda)x_n\to0$ as $n\to\infty$.  We decompose $x_n$ as follows:
\[
  x_n = u_n + v_n, \qquad u_n\in L,\; v_n\perp L.
\]
The projection onto $L$ is weakly continuous and $L$ is finite-dimensional
by assumption \textbf{(F3)};
therefore the sequence $(u_n)_{n\in \mathbb N}$  converges
strongly in $H$ to $0$, $A_0 u_n \to 0$ and $\|v_n\| \to 1$
as $n\to \infty$.  We obtain
\begin{align*}
  \bigl|\langle T(\lambda)x_n , x_n\rangle\bigr|
  & = \bigl|\frt (\lambda) [x_n]\bigr|
  = \bigl|\lambda^2 + \lambda \frd [x_n] + \fra_0 [x_n]\bigr| \\[1ex]
  & \ge \fra_0 [x_n] - \bigl(|\lambda^2| + |\lambda \frd [x_n]|\bigr)
  \ge (1-b)\fra_0 [x_n] - (|\lambda^2| + a)\\[1ex]
  & =  (1-b)\biggl(\fra_0 [u_n] + \fra_0 [v_n]
  - \frac{|\lambda|^2 + a}{1-b}
  \biggr).
\end{align*}
As $n \to \infty$, we have $\fra_0 [u_n]\to 0$, and
$\fra_0 [v_n] \ge \bigl(\frac{|\lambda|^2 +a}{1-b} +1\bigr)\|v_n\|^2$ holds
for every $n\in\NN$.
Hence
\[
  \liminf_{n\to \infty} \bigl|\langle T(\lambda)x_n , x_n\rangle\bigr|  \geq (1-b) >0,
\]
which is a contradiction.  Therefore $0\notin\sess(T(\lambda))$.

Finally, assume that $\gamma_0>0$.
Suppose that there are infinitely many eigenvalues of $\mathcal A$ in
$\Phi_{\gamma_0}\setminus \mathbb R$.  Since $\Phi_{\gamma_0}\setminus \mathbb R$
is a bounded set, there exists a sequence
of non-real eigenvalues of $\mathcal A$ which converges.
However, this contradicts Proposition~\ref{Ilmenau}.  Hence the last statement
is proved.
\end{proof}

In the following we prove variational principles for real eigenvalues
of $\mathcal A$ or, what is equivalent (see Proposition \ref{propdiskrho}), of $T$.
To this end we introduce functionals $p_+$ and $p_-$ so that $p_+$ serves
as generalized Rayleigh functional for $T$ on appropriate intervals.
For fixed $x\in H_{\frac{1}{2}}\backslash\{0\}$ consider the equation
\begin{equation}\label{scalar_qu_eq0}
  \frt(\lambda)[x] = \lambda^2\|x\|^2 + \lambda\frd[x] + \fra_0[x] = 0
\end{equation}
as an equation in $\lambda$.

\begin{definition}\label{Gustav2AAA}
If \eqref{scalar_qu_eq0} for $x\in H_{\frac12}\setminus\{0\}$
 does not have a real solution, then we set
\[
  p_+(x) \defeq -\infty, \qquad p_-(x) \defeq +\infty.
\]
Otherwise, we denote the solutions of \eqref{scalar_qu_eq0} by $p_\pm(x)$:
\begin{equation}\label{scalar_qu_eq17}
\begin{aligned}
  p_\pm(x)
  \defequ \frac{-\frd[x] \pm \sqrt{\bigl(\frd[x]\bigr)^2 - 4\|x\|^2\fra_0[x]}\,}{2\|x\|^2}
  \\[1ex]
  &= \frac{-\langle D x,x\rangle_{H_{-\frac{1}{2}}\times
  H_{\frac{1}{2}}} \pm \sqrt{\langle D x,x\rangle_{H_{-\frac{1}{2}}\times
  H_{\frac{1}{2}}}^2-4 \|x\|^2 \|A_0^{1/2} x\|^2}\,}{2 \|x\|^2}\,.
\end{aligned}
\end{equation}
Note that the values of $p_+(x)$ and $p_-(x)$ belong to $(-\infty,0)\cup\{\pm\infty\}$. Set
\begin{align}
  \cD^* \defequ \bigl\{x\in H_{\frac{1}{2}}\setminus\{0\}\mid
  \exists\,\lambda\in\RR \text{ such that } \frt(\lambda)[x]=0\bigr\} \notag\\[1ex]
  &= \bigl\{x\in H_{\frac12}\setminus\{0\} \mid p_\pm(x) \text{ are finite}\bigr\}
  \notag\\[1ex]
  &= \Bigl\{x\in H_{\frac12}\setminus\{0\} \mid
  \frd[x] \ge 2 \|x\|\,\sqrt{\fra_0[x]}\Bigr\}
  \label{Dstar}
\end{align}
and define
\begin{equation}\label{def_alpha}
  \alpha \defeq \begin{cases}
    \displaystyle \max\biggl\{\sup_{x\in\cD^*} p_-(x), -\frac{1}{\gamma _0}\biggr\}
    & \text{if } \gamma_0 > 0, \\[3ex]
    \displaystyle \sup_{x\in\cD^*} p_-(x)
    & \text{if } \gamma_0 = 0,
  \end{cases}
\end{equation}
where we set $\sup_{x\in\cD^*}p_-(x)=-\infty$
if $\cD^*=\emptyset$.
\end{definition}

We collect some of the properties of $p_+$, $p_-$ and $\alpha$
in the following lemma.
Note that $\gamma>0$ if and only if $\frd\ne0$.

\begin{lemma}\label{lem5.5}
Assume that $\frd\ne0$.  Then
\[
  p_\pm(x) < -\frac{1}{\gamma} \qquad\text{for}\;\; x\in\cD^*,
\]
and hence
\[
  \alpha \le -\frac{1}{\gamma}.
\]
\end{lemma}

\begin{proof}
The assumption $\frd\ne0$ implies that $\gamma>0$.
Let $x\in\cD^*$.  It follows from \eqref{inequ_a0_d} that
for $\lambda\in\bigl[-\frac{1}{\gamma},\infty\bigr)\setminus\{0\}$,
\begin{align*}
  \frt(\lambda)[x] &= \lambda^2\|x\|^2+\lambda\frd[x]+\fra_0[x] \\[1ex]
  &\ge \lambda^2\|x\|^2 - \frac{\frd[x]}{\gamma} + \fra_0[x]
  \ge \lambda^2\|x\|^2 > 0.
\end{align*}
Since $\frt(0)[x]=\fra_0[x]>0$, we therefore have $\frt(\lambda)[x]>0$
for all $\lambda\in\bigl[-\frac{1}{\gamma},\infty\bigr)$.
This implies that $p_\pm(x)<-\frac{1}{\gamma}$\,.
The statement on $\alpha$ follows from this and the
inequality $-\frac{1}{\gamma_0}\le-\frac{1}{\gamma}$.
\end{proof}

In the next proposition we discuss situations
when the set $\cD^*$ is empty or non-empty.
Note that (i) in the following
proposition contains a slight improvement of
the fifth  assertion in
\cite[Theorem~3.2]{JT09}.

\begin{proposition} For the set $\cD^*$ we have the following implications.
\begin{itemize}
\item[\rm (i)]
If
\begin{equation}\label{oberg1}
  A_0^{-1/2}DA_0^{-1/2}<2A_0^{-1/2},
\end{equation}
where the inequality is understood as a relation between two
self-adjoint operators in the Hilbert space $H$ $($i.e.\
$\langle A_0^{-1/2}DA_0^{-1/2}x,x\rangle < 2\langle A_0^{-1/2}x,x\rangle$ for
all $x \in H\setminus \{0\})$, then
\[
  \cD^* = \emptyset \qquad \text{and we have} \;\;\; \sigma_p(\cA) \cap \mathbb R = \emptyset.
\]
\item[\rm (ii)]
If
\begin{equation}\label{oberg2}
  \bigl\|A_0^{-1/2}DA_0^{-1/2}\bigr\| > 2\bigl\|A_0^{-1/2}\bigr\|,
\end{equation}
where the norms are the operator norm in the Hilbert space $H$, then
\[
  \cD^* \ne \emptyset.
\]
\end{itemize}
\end{proposition}

\begin{proof}
(i) Let $x\in H_{\frac12}\setminus\{0\}$ be arbitrary and set $y\defeq A_0^{1/2}x$.
From the assumption \eqref{oberg1} we obtain that
\[
  \bigl\langle A_0^{-1/2}DA_0^{-1/2} y,y\bigr\rangle
  < 2\bigl\langle A_0^{-1/2}y,y\bigr\rangle \leq 2\|y\|\,\bigl\| A_0^{-1/2}y\bigr\|,
\]
which implies
\[
  \frd[x] = \langle D x,x\rangle_{H_{-\frac{1}{2}}\times H_{\frac{1}{2}}}
  < 2\bigl\|A_0^{1/2}x\bigr\|\,\|x\| = 2\|x\|\sqrt{\fra_0[x]}.
\]
Together with \eqref{Dstar} this shows that $x\notin\cD^*$.
Hence $\cD^* = \emptyset$.
To prove the last statement in (i), let $\lambda$ be a real eigenvalue of $\cA$
with corresponding eigenvector $\smvect{x}{y} \in \dom (\mathcal A)$.
Then
\begin{equation*}
 A_0x +  \lambda Dx + \lambda^2 x=0,
\end{equation*}
by \eqref{Steingasse}, which implies that $\frt(\lambda)[x]=0$.
The latter is not possible since $\cD^* = \emptyset$.

(ii) The number $\| A_0^{-1/2}DA_0^{-1/2}\|$ is an element of
the closure of the numerical range of the self-adjoint operator
$ A_0^{-1/2}DA_0^{-1/2}$. Therefore, there exists a sequence
$(y_n)$ in $H$ with $\|y_n\|=1$ such that
\[
  \bigl\langle A_0^{-1/2}DA_0^{-1/2}y_n,y_n\bigr\rangle
  \to \bigl\|A_0^{-1/2}DA_0^{-1/2}\bigr\|
  \qquad \text{as}\;\; n\to \infty.
\]
Assumption \eqref{oberg2} implies that
$\langle A_0^{-1/2}DA_0^{-1/2}y_{n_0},y_{n_0}\rangle >2\|A_0^{-1/2}\|$
for some $n_0\in \mathbb N$.
Set $x \defeq A_0^{-1/2}y_{n_0}$; then
\begin{align*}
  \frd[x] &= \langle D x,x\rangle_{H_{-\frac{1}{2}}\times H_{\frac{1}{2}}}
  = \bigl\langle A_0^{-1/2}DA_0^{-1/2}y_{n_0},y_{n_0}\bigr\rangle
  > 2\bigl\|A_0^{-1/2}\bigr\| \\
  &\geq 2\bigl\|A_0^{-1/2}y_{n_0}\bigr\|
  = 2\|x\|\,\bigl\|A_0^{1/2}x\bigr\|
  = 2\|x\|\sqrt{\fra_0[x]}.
\end{align*}
Now we obtain from \eqref{Dstar} that $x\in\cD^*$;
hence $\cD^*\ne\emptyset$.
\end{proof}

The following theorem is one of the main results of this paper.
Recall that an eigenvalue is called semi-simple if the algebraic and
geometric multiplicities coincide, i.e.\ if there are no Jordan chains.

\begin{theorem}\label{theo100}
Assume that  {\rm\textbf{(F1)}--\textbf{(F3)}} are satisfied.
Let $\Delta$ be an interval  with $\Delta \subset  (\alpha,0]$
and $\max \Delta =0$.
Then the set $\sigma({\mathcal A})\cap \Delta$ is either empty or consists only
of a finite or infinite sequence of isolated semi-simple eigenvalues of finite multiplicity
of $\cA$.  The case of  infinitely many eigenvalues in
$\sigma(\cA)\cap\Delta$ can occur only if
$\alpha = -\frac{1}{\gamma_0}=\inf \Delta$
and, in this case, the eigenvalues
accumulate only at~$ -\frac{1}{\gamma_0}$.

If $\sigma(\cA)\cap\Delta$ is empty, then set $N\defeq0$;
otherwise, denote the eigenvalues
of $\cA$ in $\Delta$ by $(\lambda_j)_{j=1}^N$, $N\in\NN\cup\{\infty\}$, in
non-increasing order, counted according to their multiplicities:
$\lambda_1 \ge \lambda_2 \ge \cdots$.  Then the $n$th eigenvalue $\lambda_n$,
$n\in\NN$, $n\le N$, satisfies
\begin{equation}\label{minmax_re_al}
  \lambda_n = \max_{\substack{L\subset H_{1/2} \\ \dim L = n}}
  \;\; \min_{x\in L\setminus\{0\}}
  \;\; p_+(x)
  =  \min_{\substack{L\subset H \\ \dim L = n-1}}
  \;\; \sup_{\substack{x\in H_{1/2}\setminus\{0\} \\ x \perp L}}
  \; p_+(x).
\end{equation}
If $N<\infty$, then
\begin{equation}\label{Curibau}
  \begin{aligned}
    \sup_{\substack{L\subset\cD \\ \dim L = n}}
    \;\;\min_{x\in L\setminus\{0\}}\;\; p_+(x)
    &\le \inf\Delta \\[1ex]
    \inf_{\substack{L\subset H \\ \dim L = n-1}}
    \;\; \sup_{\substack{x\in\cD\setminus\{0\} \\ x \perp L}}\; p_+(x)
    &\le \inf\Delta
  \end{aligned}
  \qquad\text{for}\; n>N \;\;\text{with}\; n\le\dim H.
  \end{equation}
\end{theorem}

\begin{proof}
Except for the semi-simplicity, the first part of Theorem \ref{theo100}
follows from Proposition \ref{Ilmenau}.
Let us next prove the second part, for which we apply Theorem~\ref{th:var_right}.
To this end, we consider the operator function $T$ defined
on $\Omega\defeq\Phi_{\gamma_0}$.
Assumption (I) in Section~\ref{sec:var} is satisfied
because of Proposition~\ref{HotelCentral} and
because $T(0)=A_0$ is a positive definite operator in $H$.
Next we show that $(\nearrow)$ is satisfied.  For $x\in H_\half\setminus\{0\}$, the function
$\lambda\mapsto\frt(\lambda)[x]$ is increasing at value zero on $\Delta$ because it is convex
and a zero in $(\alpha,0]$ is the greater one of the two zeros of that function
by the definition of $\alpha$ (note that a double-zero cannot lie in $(\alpha,0]$).
Hence $(\nearrow)$ is satisfied.
Moreover, $p_+$ satisfies \eqref{def_p_incr} in both cases $x\in\cD^*$ and $x\notin\cD^*$
by the definition of $p_+$. Therefore,
\[
  p(x) \defeq p_+(x), \qquad x\in H_{\frac12},
\]
is a generalized Rayleigh functional for $T$ on $\Delta$,
cf.\ Definition~\ref{StJohn}.

By Proposition~\ref{propdiskrho} the eigenvalues and their geometric
multiplicities of $T$ and $\mathcal A$ coincide in $\Delta$, and
the interval $\Delta'$ in Theorem~\ref{th:var_right} equals now $\Delta$.
The quantity $\kappa$ in Theorem~\ref{th:var_right} is determined as
\[
  \kappa = \dim\mathcal L_{(-\infty,0)}\bigl(T(0)\bigr) =
  \dim\mathcal L_{(-\infty,0)}(A_0)=0.
\]
Now the formulae in \eqref{minmax_re_al} and in
\eqref{Curibau} follow from \eqref{minmax_re1}, \eqref{minmax_re2},
Remark \ref{rem_minmax2} and Proposition~\ref{propdiskrho}.

Let us finally show that the eigenvalues of $\cA$ in $(\alpha,0)$ are semi-simple.
Assume that $\lambda\in(\alpha,0)$ is an eigenvalue that has a Jordan chain, i.e.\
there exist vectors $\smvect{x_0}{y_0}$, $\smvect{x_1}{y_1} \in \dom(\mathcal A)$,
both being non-zero, such that
\begin{equation}\label{Jordan}
  (\mathcal A-\lambda)\vect{x_0}{y_0} = 0, \qquad
  (\mathcal A-\lambda)\vect{x_1}{y_1} = \vect{x_0}{y_0}.
\end{equation}
It follows that $y_0=\lambda x_0$ and $x_0 \ne 0$. Moreover, we have
$x_0 \in \dom(T(\lambda))$ and $T(\lambda)x_0 =0$, cf.\ \eqref{Steingasse}.
From the second equation in \eqref{Jordan} it follows that
\[
  y_1 =x_0 + \lambda x_1 \qquad \text{and} \qquad
  A_0 x_1 + Dy_1 + \lambda y_1 = -\lambda x_0.
\]
Substituting for $y_1$ we obtain
\[
  -\left( \lambda^2 + \lambda D +A_0 \right)x_1 = (2\lambda+D)x_0
\]
and hence, by \eqref{Wuppi5} and the symmetry of $\frt(\lambda)$ for
real $\lambda$,
\begin{align*}
  \bigl\langle (2\lambda+D)x_0,x_0\bigr\rangle_{H_{-\frac12}\times H_{\frac12}}
  &= - \frt(\lambda)[x_1,x_0] = -\overline{ \frt(\lambda)[x_0,x_1]} \\[1ex]
  &= -\overline{\bigl\langle T(\lambda)x_0,x_1\bigr\rangle_{H_{-\frac12}\times H_{\frac12}}}
  = 0,
\end{align*}
where we used that $x_0 \in \ker T(\lambda)$.
The left-hand side of this equation is equal to $\frt'(\lambda)[x_0]$,
which is positive because $\lambda\in(\alpha,0)$ and there is no double-zero
of  $\lambda \mapsto \frt(\lambda)[x_0]$ in $(\alpha,0]$.
This is a contradiction and hence $\lambda$ is semi-simple.
\end{proof}


The next proposition provides a sufficient condition for the existence of
eigenvalues in the interval $\bigl(-\frac{1}{\gamma_0},0\bigr)$.

\begin{proposition}\label{burggarten}
Assume that {\rm\textbf{(F1)}--\textbf{(F3)}} are satisfied and that $\gamma_0>0$.
If
\begin{equation}\label{clouds}
  \sigma\Bigl(A_0^{-1/2}DA_0^{-1/2}-\frac{1}{\gamma_0}A_0^{-1}\Bigr)\cap(\gamma_0,\infty)\ne\emptyset,
\end{equation}
then
\begin{equation}\label{sky}
  \sigma(\cA)\cap\Bigl(-\frac{1}{\gamma_0},0\Bigr)\ne\emptyset.
\end{equation}
\end{proposition}

\begin{proof}
Define the following operator function
\[
  R(\lambda) \defeq A_0^{-1/2}DA_0^{-1/2}+\lambda A_0^{-1}+
  \frac{1}{\lambda}I,
  \qquad \lambda\in\RR\setminus\{0\},
\]
whose values are bounded operators in $H$.
Assumption \eqref{clouds} implies that
\[
  \max\sigma\Bigl(R\Bigl(-\frac{1}{\gamma_0}\Bigr)\Bigr)
  = \max\sigma\Bigl(A_0^{-1/2}DA_0^{-1/2}-\frac{1}{\gamma_0}A_0^{-1}-\gamma_0I\Bigr) > 0.
\]
On the other hand, for $\lambda<0$,
\[
  \max\sigma\bigl(R(\lambda)\bigr) \le \gamma+\frac{1}{\lambda} \to -\infty
  \qquad\text{as}\;\;\lambda\to0-.
\]
Since $\max\sigma(R(\lambda))$ is continuous in $\lambda$
(see, e.g.\ \cite[Theorem~V.4.10]{K}),
there exists a $\lambda_0\in\bigl(-\frac{1}{\gamma_0},0\bigr)$ such that
$\max\sigma(R(\lambda_0))=0$.  The compactness of $A_0^{-1}$ implies that
\[
  \max\sess\bigl(R(\lambda_0)\bigr) = \gamma_0+\frac{1}{\lambda_0} < 0.
\]
Hence $0\in\sigma_{\rm p}(R(\lambda_0))$, i.e.\ there exists a $y\in H\setminus\{0\}$
such that
\[
  A_0^{-1/2}DA_0^{-1/2}y+\lambda_0A_0^{-1}y+\frac{1}{\lambda_0}y = 0.
\]
Applying $A_0^{1/2}$ to both sides, multiplying by $\lambda_0$ and
setting $x\defeq A_0^{-1/2}y$ we obtain that
\[
  \lambda_0^2x+\lambda_0Dx+A_0x = 0.
\]
This, together with \eqref{BenianCourt}, implies \eqref{sky}.
\end{proof}

The converse of Proposition~\ref{burggarten} is not true, i.e.\ \eqref{sky}
does not imply \eqref{clouds}.  This can be seen from the following example.
Let $H=\ell^2$ and define the operators $A_0$ and $D$ by
\[
  (A_0x)_n = nx_n, \qquad
  (Dx)_n = \begin{cases}
    2x_1, & n=1, \\[1ex]
    \dfrac{n}{2}x_n, & n\ge2,
  \end{cases}
\]
where $x=(x_n)_{n=1}^\infty$.  Then $\gamma_0=\frac{1}{2}$,
\[
  \sigma\Bigl(A_0^{-1/2}DA_0^{-1/2}-\frac{1}{\gamma_0}A_0^{-1}\Bigr)
  = \Bigl\{0,\frac12\Bigr\} \cup \Bigl\{\frac12-\frac{2}{n}\mid n\in\NN, n\ge 2\Bigr\},
\]
which is disjoint from $(\gamma_0,\infty)$.  However, $-1$ is an eigenvalue of $T$
with eigenvector $(1,0,0,\dots)$.

\medskip

With the help of the form $\frt(\lambda)$ it is shown in the following
proposition that a certain triangle belongs to the resolvent set
of $\mathcal A$; see Figure~\ref{figtriangle}.
This complements \cite[Theorem~3.2]{JTW}, where it was shown
that the open disc around zero with radius
\[
  r = \frac{2}{\gamma+\sqrt{\gamma^2+4\|A_0^{-1}\|}\,}\,,
\]
belongs to $\rho(\mathcal A)$; note that $r<\frac{1}{\gamma}$.

\begin{figure}[ht]
\begin{center}
\setlength\unitlength{1cm}
\begin{picture}(12,5)(-6,-2.5)
\put(-3,0){\line(1,0){0.95}}
\put(-1.95,0){\vector(1,0){5}}
\put(0,-2.2){\vector(0,1){4.4}}
\put(0,0){\line(-1,1){1.95}}
\put(0,0){\line(-1,-1){1.95}}
\put(-2,0.05){\line(0,1){1.9}}
\put(-2,-0.05){\line(0,-1){1.9}}
\put(-2,0){\circle{0.1}}
\put(-2,2){\circle{0.1}}
\put(-2,-2){\circle{0.1}}
\put(-2,1.4){\line(1,2){0.2}}
\put(-2,0.8){\line(1,2){0.4}}
\put(-2,0.2){\line(1,2){0.6}}
\put(-2,-0.4){\line(1,2){0.8}}
\put(-2,-1){\line(1,2){1.0}}
\put(-2,-1.6){\line(1,2){1.2}}
\put(-1.8,-1.8){\line(1,2){1.2}}
\put(-1.2,-1.2){\line(1,2){0.8}}
\put(-0.6,-0.6){\line(1,2){0.4}}
\put(2.5,-0.4){$\Re z$}
\put(0.2,2.0){$\Im z$}
\put(-2.6,-0.4){$-\frac{1}{\gamma}$}
\put(-3.3,1.9){$-\frac{1}{\gamma}+\frac{i}{\gamma}$}
\put(-3.3,-2.1){$-\frac{1}{\gamma}-\frac{i}{\gamma}$}
\end{picture}
\caption{The region on the left-hand side of \eqref{diskrho},
which is contained in $\rho(\cA)$;
the three circles indicate the numbers
$-\frac{1}{\gamma}$, $-\frac{1}{\gamma}\pm\frac{i}{\gamma}$,
which, in general, do not belong to $\rho(\cA)$.}
\label{figtriangle}
\end{center}
\end{figure}
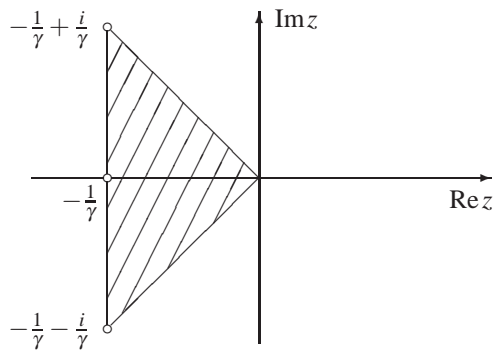

\begin{proposition}\label{propdiskrho2}
Assume that $\frd\ne0$.  Then
\begin{equation}\label{diskrho}
  \biggl\{z\in\CC\biggm| z=0\;\;\text{or}\;\;-\frac{1}{\gamma}\le\Re z<0,\,
  \arg z\in\biggl[\frac{3\pi}{4},\frac{5\pi}{4}\biggr]\biggr\}
  \Big\backslash \biggl\{-\frac{1}{\gamma}\,,-\frac{1}{\gamma}\pm\frac{i}{\gamma}\biggr\}
  \subset \rho(\mathcal A)
\end{equation}
where $\gamma$ is defined in \eqref{Elbersfeld}. If, in addition,
$\gamma \ne \gamma_0$, then also $-\frac{1}{\gamma}\in\rho(\mathcal A)$.
\end{proposition}

\begin{proof}
Since $\frd\ne0$, we have $D\ne 0$ and $\gamma >0$, see \eqref{Elbersfeld}.
Let $\lambda$ be either in the set on the left-hand side of \eqref{diskrho}
or let $\lambda=-\frac{1}{\gamma}$ and assume that $\gamma\ne\gamma_0$
in the latter case.
Suppose that $\lambda\in\sigma(\cA)$.
By Proposition \ref{Ilmenau} the
set on the left-hand side of \eqref{diskrho} is disjoint from $\sess(\cA)$,
and $-\frac{1}{\gamma}\notin\sess(\cA)$ if $\gamma\ne\gamma_0$.
Hence $\lambda$ is an eigenvalue of $\cA$.
By \eqref{Steingasse} there exists an $x\in H_{\frac12}\setminus\{0\}$
such that $\frt(\lambda)[x]=0$.
We have $\Re(\lambda)\ge-\frac{1}{\gamma}$ and $\Re(\lambda^2)\ge0$,
where at least one of the two inequalities is strict.
Using \eqref{inequ_a0_d} we therefore obtain
\begin{align*}
  0 &= \re\bigl(\frt(\lambda)[x]\bigr)
  = \Re(\lambda^2)\|x\|^2 + (\Re\lambda)\frd[x] + \fra_0[x] \\[1ex]
  &\ge \Re(\lambda^2)\|x\|^2 + (\Re\lambda)\gamma\,\fra_0[x] + \fra_0[x] \\[1ex]
  &> -\frac{1}{\gamma}\cdot\gamma\,\fra_0[x] + \fra_0[x] = 0,
\end{align*}
which is a contradiction.  Hence $\lambda\in\rho(\cA)$.
\end{proof}

One can easily construct examples with eigenvalues $\lambda$ of  $\mathcal A$
satisfying $\arg\lambda\in\bigl(\frac{\pi}{2},\frac{3\pi}{4}\bigr)$
and $|\,$Re$\,\lambda|<\frac{1}{\gamma}$\,.  For example, let $A_0$ be
a positive definite operator with compact resolvent and smallest eigenvalue $1/2$.
For the choice $D=A_0$, we have $\gamma =1$
and $\lambda_0=-\frac{1}{4}+i\frac{\sqrt{7}}{4}$ is an eigenvalue of $\mathcal A$
which satisfies $\Re\lambda_0=-1/4>-\frac{1}{\gamma}=-1$ and
$\arg\lambda_0\in\bigl(\frac{\pi}{2},\frac{3\pi}{4}\bigr)$.

Another application of Theorem \ref{theo100} results in interlacing
properties of eigenvalues of two different second-order problems
with coefficients that satisfy a specific order relation. This
is the content of the following theorem.

\begin{theorem}\label{comparison1}
Let the forms $\fra_0$, $\hat\fra_0$, $\frd$ and $\hat\frd$ in the Hilbert space $H$
be given so that $\fra_0$, $\frd$ and $\hat\fra_0$, $\hat\frd$, respectively,
satisfy assumptions {\rm\textbf{(F1)}--\textbf{(F3)}}.
Assume that $\cD(\fra_0)=\cD(\hat\fra_0)$ and
\begin{equation}\label{Bischi}
  \fra_0[x] \ge \hat\fra_0[x], \qquad \frd[x] \le \hat\frd[x]
  \qquad\text{for}\;\;x\in\cD(\fra_0).
\end{equation}
Let $\hat\cA$, $\hat\delta$, $\hat\gamma$, $\hat\delta_0$, $\hat\gamma_0$,
$\hat\frt$, $\hat p_{\pm}$, and  $\hat\alpha$ be defined as in
\eqref{spaetabends}--\eqref{spaetabends2},
\eqref{Elbersfeld}, \eqref{Elbersfeldberg}, \eqref{Kuchen}, \eqref{Wuppi5},
and \eqref{scalar_qu_eq17}--\eqref{def_alpha}, respectively,
where $\fra_0$ is replaced by $\hat\fra_0$ and $\frd$  by $\hat\frd$.
Then we have
\begin{equation}\label{gammagammahat}
  \gamma \le \hat\gamma, \qquad \gamma_0 \le \hat\gamma_0, \qquad
  \delta \le \hat\delta, \qquad \delta_0 \le \hat\delta_0.
\end{equation}
Let
\[
  \Delta \defeq (a,0] \qquad\text{with}\quad a\ge\max\{\alpha,\hat\alpha\}.
\]
Assume now that $\sigma(\cA)\cap \Delta$ is non-empty;
then also $\sigma(\hat\cA)\cap \Delta$ is non-empty.
Let $(\lambda_n)_{n=1}^{N}$ and $(\hat\lambda_n)_{n=1}^{\hat N}$,
$N, \hat N\in \mathbb N \cup \{\infty\}$, be the eigenvalues of $\cA$ and $\hat\cA$,
respectively, in the interval $\Delta$, both arranged in non-increasing order
and counted according their multiplicities.
Then $N\le\hat N$ and
\begin{equation}\label{rentier}
  \lambda_{n} \le \hat\lambda_{n}
  \qquad \text{for}\;\;n\in\NN,\,n\le N.
\end{equation}
\end{theorem}

\begin{proof}
The inequalities in \eqref{gammagammahat} follow from \eqref{Bischi}, \eqref{deltainf},
\eqref{gammasup}, \eqref{delta0inf} and \eqref{Elbersfeldberg}, e.g.\
\[
  \gamma = \sup_{y\in H_{1/2}\setminus\{0\}}\frac{\frd[y]}{\fra_0[y]}
  \le \sup_{y\in H_{1/2}\setminus\{0\}}\frac{\hat\frd[y]}{\hat\fra_0[y]} = \hat\gamma.
\]
The relations in \eqref{Bischi} imply that
\[
  \frt(\lambda)[x] \ge \hat\frt(\lambda)[x],
  \qquad x\in H_{\frac12},\;\lambda\in(-\infty,0].
\]
It follows from \eqref{Trip2} that $p_+(x)\le\hat p_+(x)$ for $x\in\hhalf\setminus\{0\}$
and hence
\begin{equation}\label{Stockholm}
  \mu_n \defeq \sup_{\substack{L\subset H_{1/2} \\ \dim L = n}}
  \;\;\min_{x\in L\setminus\{0\}}\;\; p_+(x)
  \le \sup_{\substack{L\subset H_{1/2} \\ \dim L = n}}
  \;\;\min_{x\in L\setminus\{0\}}\;\; \hat p_+(x)
  \eqdef \hat\mu_n.
\end{equation}
Assume that $\cA$ has at least $m$ eigenvalues in $\Delta$.  Then, by Theorem \ref{theo100}, $\lambda_m=\mu_m>a$.
If $\hat\cA$ had less than $m-1$ eigenvalues in $\Delta$, then $\hat\mu_m\le a$
by \eqref{Curibau}, which is a contradiction to \eqref{Stockholm}.
Hence the implication
$\sigma(\cA)\cap\Delta\ne\emptyset\;\Rightarrow\;\sigma(\hat\cA)\cap\Delta\ne\emptyset$
and the inequality $N\le\hat N$ are true.
Finally, the inequality in \eqref{rentier} follows from \eqref{minmax_re_al}
and \eqref{Stockholm}.
\end{proof}

\section{Example: beam with damping}
\label{sec:example}

We consider a beam of length $1$ and study transverse vibrations only.
Let $u(r,t) $ denote the deflection of the beam  from its rigid body motion
 at time $t$ and position $r$. We consider
 for the beam deflection  a damping model which
leads to the following description of the vibrations
where $a_0>0$ is a real constant and $d\in C^1[0,1]$ with $\min_{r\in[0,1]} d(r)>0$:
\begin{equation}\label{beamIlm}
   \frac{\partial ^2 u} {\partial t^2} +
   a_0 \frac{\partial ^4 u } {\partial r^4}
   + \frac{\partial^2 }{ \partial t\partial r }\left[ d \frac{\partial u}{\partial r }\right]
   = 0,
   \hspace{2em} r \in (0,1), \, t > 0.
\end{equation}
Assuming that the beam is pinned, free to rotate and does not experience any
torque at both ends, we have for all $t>0$ the following boundary conditions
\begin{equation} \label{bcsIlm}
  u\big|_{r=0} = u\big|_{r=1} = \frac{\partial ^2 u } {\partial r^2 }\bigg|_{r=0}
  = \frac{\partial ^2 u } {\partial r^2 }\bigg|_{r=1} = 0.
\end{equation}
We consider the partial differential equation
\eqref{beamIlm}--\eqref{bcsIlm} as a
second-order problem in the Hilbert space $H = L^2 (0,1)$.
In order to formulate this beam equation as in \eqref{sys}, we introduce
the forms $\fra_0$ and $\frd$ defined for
$x,y$ from the form domains $\cD(\fra_0) =\cD(\frd)=
H^2(0,1)\cap H_0^1(0,1)$ as
\[
  \fra_0[x,y] \defeq a_0\int_0^1 x''(r)\overline{y''(r)} \rd r
  \qquad \text{and} \qquad
  \frd[x,y] \defeq \int_0^1d(r) x'(r)\overline{y'(r)} \rd r.
\]
Then \eqref{beamIlm}--\eqref{bcsIlm} corresponds to
\[
  \langle \ddot{u}(t),y\rangle + \fra_0[u(t),y] + \frd[\dot{u}(t),y]=0
  \qquad \text{for all } y \in \cD(\fra_0) =\cD(\frd).
\]

Set
\[
  \dmin \defeq \min_{r\in[0,1]}d(r), \qquad \dmax\defeq \max_{r\in[0,1]}d(r).
\]
For $x \in \cD(\fra_0)$ we have
\[
  \fra_0[x] = a_0 \langle x'', x''\rangle \geq a_0 \pi^4 \|x\|^2,
\]
which shows \textbf{(F1)}.  Using again $\|x''\| \geq \pi^2 \|x\|$
we obtain for $x\in \cD(\fra_0)$ that
\begin{equation*}
\begin{split}
  \fra_0[x] & = a_0\|x''\|^2 \geq a_0 \pi^2 \|x''\|\|x\| \geq
  a_0 \pi^2 \left| \int_0^1 x''(r) \overline{x(r)} \rd r\right| \\
  & =a_0 \pi^2 \int_0^1 \bigl|x'(r)\bigr|^2\rd r
  \geq \frac{a_0\pi^2}{\dmax}\int_0^1 d(r)\bigl|x'(r)\bigr|^2\rd r
  =  \frac{a_0\pi^2}{\dmax} \frd[x].
\end{split}
\end{equation*}
Thus  \textbf{(F2)} holds. In order to show  \textbf{(F3)}
we introduce the operator $A_0$ associated with $\fra_0$ via
the the First Representation Theorem \cite[Theorem~VI.2.1]{K} as in
\eqref{reprA0}. It is easy to see that $A_0$ has the form
\begin{align*}
  A_0 =  a_0 \frac{\rd^4}{\rd r^4},
  \quad
  \dom(A_0) = \left\{ z \in \cD(\fra_0) \mid  z''
  \in \cD(\fra_0) \right\}.
\end{align*}
Obviously, $A_0$ satisfies assumption \textbf{(F3)}. We define the Hilbert space
$H_{\frac{1}{2}}$ as in \eqref{Cambridge11}; then
$H_{\frac{1}{2}} = \cD(\fra_0) = \cD(\frd)$.
Moreover, we define the damping operator as
\begin{align*}
   D \defeq -\frac{\rd}{\rd r} \left[d\frac{\rd}{\rd r}\right].
\end{align*}
Due to the fact that $d\in C^1[0,1]$, $D$ is a linear bounded operator from $ H_{\half}$ to $H$. For $x\in H_\half$ we have
\begin{equation*}
   \langle Dx,x\rangle = \langle dx', x'\rangle = \frd[x].
\end{equation*}
Since $DA_0^{-1/2}$ is a bounded operator in $H$
and $A_0^{-1/2}$ is a compact operator in $H$, we see that
$A_0^{-1/2}DA_0^{-1/2}$ is a compact operator in $H$.
From this we obtain
\begin{equation*}
  \sess\bigl(A_0^{-1/2}DA_0^{-1/2}\bigr)=\{0\}
\end{equation*}
and hence $\gamma_0=\delta_0=0$.  This, together with Proposition~\ref{Ilmenau},
yields
\begin{equation}\label{eqn1-Ilm}
  \sess(\cA) =\emptyset.
\end{equation}
Finally, we apply the results of this paper to the damped beam equation.

\begin{theorem}
Assume that
\begin{equation}\label{Bayern}
  \dmin^2 \ge 4a_0.
\end{equation}
Then $\cD^* \ne \emptyset$ $($cf.\ \eqref{Dstar}$)$ and
the number $\alpha$ from \eqref{def_alpha} satisfies $\alpha\le-\dmin\frac{\pi^2}{2}$.
The set 
\[
  \sigma(\cA)\cap\Bigl(-\frac{\dmin\pi^2}{2},0\Bigr)
\]
is non-empty and consists only of a finite  sequence
of isolated semi-simple eigenvalues of finite multiplicity
of $\cA$ counted according to their multiplicities:
$\lambda_1 \geq \lambda_2 \geq \ldots \geq \lambda_N$
for some $N\in \mathbb N$. The $n$th eigenvalue $\lambda_n$,
$1\leq n\leq N$, satisfies \eqref{minmax_re_al} in Theorem~\ref{theo100}
and the following inequalities:
\begin{equation}\label{beamupperbound}
  \lambda_n \le \frac{-\dmax+\sqrt{\dmax^2-4a_0}\,}{2}\cdot\pi^2n^2,
  \qquad 1\le n\le N,
\end{equation}
and
\begin{equation}\label{beamlowerbound}
  \lambda_n \ge \frac{-\dmin+\sqrt{\dmin^2-4a_0}\,}{2}\cdot\pi^2n^2,
  \qquad n\in\NN\;\;\text{such that}\;\;n^2\le \frac{1}{1-\sqrt{1-\frac{4a_0}{\dmin^2}}\,}\,.
\end{equation}
\end{theorem}

Note that the inequality in \eqref{beamlowerbound} for $\lambda_n$ holds at least for $n=1$.


\begin{proof}
We introduce the forms $\frdmin$ and $\frdmax$ by
\[
  \frd_{\rm min/max}[x,y] \defeq d_{\rm min/max}\int_0^1 x'(r)\overline{y'(r)}\rd r,
  \qquad x,y\in\hhalf,
\]
the form polynomials $\frtmin$ and $\frtmax$ by
\[
  \frt_{\rm min/max}(\lambda)[x,y] \defeq \lambda^2\langle x,y\rangle
  + \lambda\frd_{\rm min/max}[x,y] + \fra_0[x,y], \qquad x,y\in\hhalf,
\]
and the corresponding operator functions $T_{\rm min}$ and $T_{\rm max}$  as in Proposition \ref{HotelCentral}.
Let $S\defeq-\frac{\rd^2}{\rd r^2}$ in $L^2(0,1)$
with domain $\cD(S)=H^2(0,1)\cap H_0^1(0,1)$,
which has spectrum $\sigma(S)=\{n^2\pi^2\mid n\in\NN\}$.
Since we can write
\[
  T_{\rm min/max}(\lambda) = \lambda^2+\lambda d_{\rm min/max}S+a_0S^2,
\]
we can use the spectral mapping theorem to obtain
\begin{align}
  \sigma(\Tmin) &= \bigl\{\lambda\in\CC\mid \lambda^2+\lambda\dmin n^2\pi^2+a_0n^4\pi^4=0
  \;\;\text{for some}\;\;n\in\NN\bigr\}
    \notag\\[1ex]
  &= \Biggl\{\frac{-\dmin\pm\sqrt{\dmin^2-4a_0}}{2}\cdot n^2\pi^2\,\bigg|\; n\in\NN\Biggr\}
  \subset (-\infty,0).
    \label{lambdanmin}
\end{align}
In a similar way one obtains a description of $\sigma(\Tmax)$.

Define $p_\pm$, $\pmin_\pm$, $\pmax_\pm$, $\cD^*$, $\cD^*_{\rm min}$, $\cD^*_{\rm max}$,
$\alpha$, $\alpha_{\rm min}$, $\alpha_{\rm max}$ as in Definition~\ref{Gustav2AAA}
corresponding to $T$, $\Tmin$ and $\Tmax$, respectively.
Denote by $e_1$ the eigenvector to the smallest eigenvalue, $\pi^2$, of $S$
with $\|e_1\|=1$, i.e.\ $e_1=\sqrt{2}\sin(\pi\, \cdot)$
and $Se_1=\pi^2 e_1$.
It follows from \eqref{Bayern} that
\begin{align*}
  \frd[e_1]-2\|e_1\|\sqrt{\fra_0[e_1]}
  &\ge \frdmin[e_1]-2\sqrt{\fra_0[e_1]} \\[1ex]
  &= \dmin\langle Se_1,e_1\rangle-2\sqrt{a_0}\|Se_1\|
  = \dmin\pi^2-2\sqrt{a_0}\pi^2
  \ge 0,
\end{align*}
which by \eqref{Dstar} implies that $\cD^*\ne\emptyset$.
Since $\gamma_0=0$, we have
\begin{align*}
  \alpha &= \sup_{x\in\cD^*}p_-(x)
  = \sup_{x\in\cD^*}\frac{-\frd[x]-\sqrt{\bigl(\frd[x]\bigr)^2-4\|x\|^2\fra_0[x]}\,}{2\|x\|^2}
    \\[1ex]
  &\le \sup_{x\in\cD^*}\frac{-\frd[x]}{2\|x\|^2}
  \le \sup_{x\in H_{1/2}}\frac{-\frdmin[x]}{2\|x\|^2}
  = -\inf_{x\in H_{1/2}}\frac{\dmin\|x'\|^2}{2\|x\|^2}
  = -\frac{\dmin\pi^2}{2}\,.
\end{align*}
In the same way one obtains that $\alpha_{\rm min},\alpha_{\rm max}\le-\frac{\dmin\pi^2}{2}$.

Set $\Delta\defeq\bigl(-\frac{\dmin\pi^2}{2},0\bigr]$ and let 
$(\lambdamin_n)_{n=1}^{\Nmin}$ and $(\lambdamax_n)_{n=1}^{\Nmax}$ be the
eigenvalues of 
$\Tmin$ and $\Tmax$, respectively, in the interval $\Delta$ ordered non-increasingly
and counted with multiplicities.
We can apply Theorem~\ref{comparison1} to the pairs $\Tmin$, $T$ and
$T$, $\Tmax$, which implies that $\Nmin\le N\le\Nmax$ and
\begin{equation}\label{compBeam}
\begin{alignedat}{2}
  \lambdamin_n &\le \lambda_n, \qquad & & 1\le n\le \Nmin, \\[1ex]
  \lambda_n &\le \lambdamax_n, \qquad & & 1\le n\le N.
\end{alignedat}
\end{equation}
It follows from \eqref{lambdanmin} that
\[
  \lambdamin_n = \frac{-\dmin+\sqrt{\dmin^2-4a_0}\,}{2}\cdot n^2\pi^2, \qquad
  \lambdamax_n = \frac{-\dmax+\sqrt{\dmax^2-4a_0}\,}{2}\cdot n^2\pi^2.
\]
Moreover, $\Nmin$ is the largest positive integer such that
\[
  \frac{-\dmin+\sqrt{\dmin^2-4a_0}}{2}\cdot\Nmin^2\pi^2 \ge -\frac{\dmin\pi^2}{2}\,,
\]
where the latter inequality is equivalent to
\begin{equation}\label{Nmin}
  \Nmin^2 \le \frac{\dmin}{\dmin-\sqrt{\dmin^2-4a_0}}
  = \frac{1}{1-\sqrt{1-\frac{4a_0}{\dmin^2}}\,}\,.
\end{equation}
Now the inequalities in \eqref{compBeam} imply \eqref{beamupperbound}
and \eqref{beamlowerbound}.
Since the right-hand side of \eqref{Nmin} is greater than or equal to $1$,
we have $N\ge\Nmin\ge1$.  Hence $\sigma(\cA)\cap\Delta\ne\emptyset$.
Moreover, $N$ is finite because $\sess(\cA)=\emptyset$.
\end{proof}

\section*{Acknowledgements}
Finally, the authors like to thank the anonymous referee for suggestions
that improved the exposition of the paper.



\begin{thebibliography}{20}

\bibitem{AI}
\textsc{T.\,Ya.~Azizov and I.\,S.~Iokhvidov},
\textit{Linear Operators in Spaces with an Indefinite Metric},
John Wiley \& Sons, 1989.


\bibitem{bi}
\textsc{H.\,T.~Banks and K.~Ito},
A unified framework for approximation in inverse problems for distributed parameter systems,
\textit{Control Theory and Adv.\ Tech.} \textbf{4} (1988), 73--90.

\bibitem{biw}
\textsc{H.\,T.~Banks, K.~Ito and Y.~Wang},
Well posedness for damped second order systems with unbounded input operators,
\textit{Differential Integral Equations} \textbf{8} (1995), 587--606.



\bibitem{BEL00}
\textsc{P.~Binding, D.~Eschw\'e and H.~Langer},
Variational principles for real eigenvalues of self-adjoint operator pencils,
\emph{Integral Equations Operator Theory} \textbf{38} (2000), 190--206.

\bibitem{B}
\textsc{J.~Bognar},
\textit{Indefinite Inner Product Spaces},
Springer, 1974.

\bibitem{CLL}
\textsc{S.~Chen, K.~Liu and Z.~Liu},
Spectrum and stability for elastic systems with global or local Kelvin--Voigt damping,
\textit{SIAM J.\ Appl.\ Math.} \textbf{59} (1998), 651--668.

\bibitem{D55}
\textsc{R.\,J.~Duffin},
A minimax theory for overdamped networks,
\textit{J.\ Rational Mech.\ Anal.} \textbf{4} (1955), 221--233.

\bibitem{EE}
\textsc{D.\,E.~Edmunds and W.\,D.~Evans},
\textit{Spectral Theory and Differential Operators},
Oxford University Press, 1987.



\bibitem{EL}
\textsc{D.~Eschw\'e and M.~Langer},
Variational principles for eigenvalues of self-adjoint operator functions,
\textit{Integral Equations Operator Theory} \textbf{49} (2004), 287--321.

\bibitem{HS}
\textsc{R.\,O.~Hryniv and A.\,A.~Shkalikov},
Exponential stability of semigroups related to operator models in mechanics,
\textit{Math.\ Notes} \textbf{73} (2003), 657--664.

\bibitem{JMT}
\textsc{B.~Jacob, K.~Morris and C.~Trunk},
Minimum-phase infinite-dimensional second-order systems,
\textit{IEEE Transactions on Automatic Control} \textbf{52} (2007), 1654--1665.

\bibitem{jatr}
\textsc{B.~Jacob and C.~Trunk},
Location of the spectrum of operator matrices which are associated to
second order equations,
\textit{Oper.\ Matrices} \textbf{1} (2007), 45--60.

\bibitem{JT09}
\textsc{B.~Jacob and C.~Trunk},
Spectrum and analyticity of semigroups arising in elasticity theory and
hydromechanics,
\textit{Semigroup Forum} \textbf{79} (2009), 79--100.

\bibitem{JTW}
\textsc{B.~Jacob, C.~Trunk and M.~Winklmeier},
Analyticity and Riesz basis property of semigroups associated to damped vibrations,
\textit{J.\ Evol.\ Equ.} \textbf{8} (2008), 263--281.

\bibitem{K}
\textsc{T.~Kato},
\textit{Perturbation Theory for Linear Operators},
Second Edition, Springer, 1976.



\bibitem{las}
\textsc{I.~Lasiecka},
Stabilization of wave and plate equations with nonlinear dissipation
on the boundary,
\textit{J.\ Differential Equations} \textbf{79} (1989), 340--381.



\bibitem{T}
\textsc{C.~Trunk},
Spectral theory for operator matrices related to models in mechanics,
\textit{Math.\ Notes} \textbf{83} (2008), 843--850.

\bibitem{TWII}
\textsc{M.~Tucsnak and G.~Weiss},
How to get a conservative well-posed system out of thin air, Part II,
\textit{SIAM J.\ Control Optim.} \textbf{42} (2003), 907--935.

\bibitem{voss15} H.~Voss,
Variational principles for eigenvalues of nonlinear eigenproblems,
in: \textit{Numerical Mathematics and Advanced Applications -- ENUMATH 2013},
Lecture Notes in Computational Science and Engineering, vol.~103
(2015), 305--313.

\bibitem{wein}
\textsc{A.~Weinstein and W.~Stenger},
Methods of Intermediate Problems for Eigenvalues,
Academic Press, 1972.

\bibitem{WT}
\textsc{G.~Weiss and M.~Tucsnak},
How to get a conservative well-posed system out of thin air, Part I,
\textit{ESAIM Control Optim.\ Calc.\ Var.} \textbf{9} (2003), 247--274.




\end{thebibliography}
\end{document}